\documentclass{amsart}
\usepackage{amsmath, amsthm, amsfonts, amssymb, txfonts, graphicx, verbatim}
\theoremstyle{plain}
\newtheorem{thrm}{Theorem}[section]
\newtheorem{lmm}[thrm]{Lemma}
\newtheorem{prpstn}[thrm]{Proposition}

\numberwithin{sblmm}{thrm}
\DeclareMathOperator*{\E}{\mathbb{E}}
\begin{document}
\author{James Maynard}
\address{Mathematical Institute, 24–-29 St Giles', Oxford, OX1 3LB}
\title{On the difference between consecutive primes}
\date{}
\thanks{Supported by EPSRC Doctoral Training Grant EP/P505216/1 }
\begin{abstract}
We show that the sum of squares of differences between consecutive primes $\sum_{p_n\le x}(p_{n+1}-p_n)^2$ is bounded by $x^{5/4+\epsilon}$ for $x$ sufficiently large and any fixed $\epsilon>0$. This reproduces an earlier result of Peck, which the author was initially unaware of. 
\end{abstract}
\maketitle
\section{Update: 16/01/2012}
The same result was obtained by Peck \cite{Peck} in his thesis in 1996. The methods used here are fundamentally the same. This work does not include any new results.

\section{Introduction and Context}
One central topic in number theory is understanding the distribution of prime numbers. When investigating the distribution of primes, it is natural to look at the gaps between them.

We let $p_n$ denote the $n^{th}$ prime number, and $d_n=p_{n+1}-p_n$ denote the $n^{th}$ prime gap.
\subsection{Average Size of Prime Gaps}\label{AverageGap}
The prime number theorem was conjectured by Gauss in 1792, and proven independently by Hadamard \cite{Hadamard} and de la Vall\'ee Poussin \cite{DeLaValleePoussin}. It states that
\begin{equation}
\pi(x)\sim \frac{x}{\log{x}}.
\end{equation}
This shows that
\begin{equation}
\E_{x\le p_n\le 2x}d_n\sim \log{x},
\end{equation}
and so the average gap between primes of size approximately $x$ is $\log{x}$. 

Since $\log x$ is small in comparison with $x$ (the size of primes we are considering), it is natural to consider how much larger $d_n$ can be than this average. The basic intuition is that prime numbers are reasonably regular, and so the difference between consecutive primes can not be unusually `large'.
\subsection{Numerical Evidence and Heuristics}
When obtaining results for prime gaps we are usually interested in large primes (primes outside of any computable range). Computable primes are not necessarily representative of all primes. (For example, Littlewood's result that $\pi(x)>$li$(x)$ infinitely often requires large primes. The first occurrence of this is well outside computational bounds). In particular, most results bounding the size of $d_n$ take the form $d_n\ll f(p_n)$ with no explicit size of the implied constant. This constant would likely dominate the bound in any computable region if it was effective and calculated.

That said, it can be interesting to look at the size of prime gaps in a computable region.
\begin{center}
\begin{tabular}{|c|c|c|}
\hline
N & $\max_{p_n\le N}d_n$ & $\max_{p_n\le N}(\log d_n)/(\log N)$\\
\hline
$10^1$	&	4 & 0.60\\
$10^2$	&	8 & 0.45\\
$10^3$	&	20 & 0.43\\
$10^4$	&	36 & 0.39\\
$10^5	$	& 72 & 0.37\\
$10^6	$	& 114 & 0.34\\
$10^7	$	& 154 & 0.31\\
$10^8	$	& 220 & 0.29\\
$10^9	$	& 282 & 0.27\\
$10^{10}$	&	354 & 0.25\\
$10^{11}$	&	464 & 0.24\\
$10^{12}$	&	540 & 0.23\\
$10^{13}$	&	674 & 0.22\\
$10^{14}$	&	804 & 0.21\\
$10^{15}$	&	906 & 0.20\\
$10^{16}$	&	1132 & 0.19\\
\hline
\end{tabular}
\end{center}
Relative to the size of the primes, the gaps between the primes of this size remain very small. Based on this very limited numerical evidence, it appears that 
\begin{equation*}
\log(d_n)/\log(p_n)\rightarrow0.
\end{equation*}
This is equivalent to the statement
\begin{equation}
d_n\ll p_n^{\epsilon}
\end{equation}
for any $\epsilon>0$.

Based on numerical evidence Legendre \cite{Legendre} conjectured in 1798 that there is always a prime between any pair of consecutive squares. Proving this requires an estimate of the strength $d_n\le 2p_n^{1/2}$. Cram\'er \cite{Cramer} and Shanks \cite{Shanks} have made stronger conjectures based on probabilistic models of the primes. Although more sophisticated models give slightly different expectations of the maximal asymptotic size of $d_n$ (as pointed out by Granville \cite{Granville}), there appears no reason to disbelieve a conjecture such as
\begin{equation}
d_n\ll (\log p_n)^{2+\epsilon}.
\end{equation}
Unfortunately even Legendre's conjecture seems beyond the current machinery for dealing with primes (even under the assumption of  strong conjectures such as the Riemann Hypothesis). A result as strong as Cram\'er's conjecture appears completely impossible to prove with the available techniques.

\subsection{Large Prime Gap Bounds}\label{LargeGaps}
Although we cannot prove Legendre's or Cram\'er's conjectures, we can still obtain non-trivial bounds on the size of $d_n$.

Bertrand's Postulate states that there is always a prime between any integer $n$ and $2n-2$. This was conjectured in 1845 by Bertrand \cite{BertrandConjecture} and proven in 1850 by Chebyshev \cite{ChebyschevBertrandProof}. There is therefore a prime between $p_n+1$ and $2p_n$, and so we must have 
\begin{equation}
d_n\le p_n.
\end{equation}
Further advancements were then made from analysing the distribution of zeroes of the Riemann Zeta function. Hoheisel \cite{Hoheisel} showed that
\begin{equation}\label{Hoheisel}
d_n\ll p_n^{32999/33000}.
\end{equation}
The exponent of $p_n$ in the right hand side has been repeatedly reduced by different authors including Heilbronn \cite{Heilbronn}, Tchudakoff \cite{Tchudakoff}, Ingham \cite{Ingham} and Huxley \cite{Huxley}. These improvements were largely down to the development of more sophisticated methods to analyse the distribution of the zeroes of $\zeta(s)$. The most recent result is due to Baker, Harman and Pintz \cite{BakerEtAl}, which shows that
\begin{equation}
d_n\ll p_n^{21/40}.
\end{equation}
Better results can be obtained if we assume conjectures about the Riemann Zeta function.

Riemann famously conjectured that all the non-trivial zeroes of $\zeta(s)$ have real part $1/2$. Cram\'er \cite{Cramer} showed that assuming the Riemann Hypothesis
\begin{equation}
d_n\ll p_n^{1/2}\log{p_n}.
\end{equation}
The density hypothesis states that the number of zeroes $N(\sigma,T)$ of $\zeta(s)$ with absolute value of imaginary part less than $T$ and real part greater than $\sigma$ satisfies $N(\sigma,T)\ll T^{2(1-\sigma)}\log^A{T}$ for some constant $A$. This follows from the Riemann Hypothesis or the Lindel\"of Hypothesis. Assuming this weaker hypothesis one can prove
\begin{equation}
d_n\ll p_n^{1/2+\epsilon}
\end{equation}
for any $\epsilon>0$. Both of these conditional results would therefore show that there is always a prime in the interval $[x,x+x^{1/2+\epsilon}]$ for $x$ sufficiently large.
\subsection{Lower Bounds on Large Prime Gaps}
One can construct sequences of consecutive composite integers to explicitly demonstrate large gaps between primes. For example, 
\begin{equation}
j+\prod_{i\le n}p_i
\end{equation}
is clearly composite for $2\le j\le p_n$. This (and small refinements) show that $d_n\ge C\log{n}$ for some constant $C$.

Westzynthius \cite{Westzynthius} showed in 1931 that by carefully sieving certain primes one can have gaps between primes which are larger than any constant multiple of the average gap $\log{p_n}$. Erd\H os \cite{Erdos} and Rankin \cite{Rankin} subsequently improved the size of this lower bound on $d_n$ using similar ideas. The best current result is due to Pintz \cite{Pintz} which states that for infinitely many integers $n$ we have
\begin{equation}
d_n\ge (2e^\gamma+o(1)) \frac{(\log{n})(\log\log{n})(\log\log\log\log{n})}{(\log\log\log{n})^2}.
\end{equation}
Note that this lower bound is only slightly larger than the average bound $\log{n}$, and is less than the upper bound of $\log^2{n}$ predicted by Cram\'er's conjecture.
\subsection{Frequency of Large Prime Gaps}
The Results of Section \ref{AverageGap} give a precise asymptotic value of the $L^1$ norm of $d_n$ from the Prime Number Theorem.

The results in Section \ref{LargeGaps} give bounds on the $L^\infty$ norm of $d_n$, but fall short of what the expected bound on gaps between primes should be, even with the assumption of the Riemann Hypothesis. It seems with the current technology we cannot hope to prove anything close to the true size of the $L^\infty$ bound.

It is therefore natural to look at the $L^2$ norm of $d_n$. Even if we cannot show that unusually large gaps do not occur, we can hope to show that the vast majority of prime gaps are much smaller and that large gaps, should they exist, are infrequent.

Selberg \cite{Selberg} proved, assuming the Riemann Hypothesis, that
\begin{equation}
\sum_{p_n\le x}d_n^2\ll x(\log{x})^3.
\end{equation}
In particular, this shows that almost all intervals $[x,x+(\log{x})^{2+\epsilon}]$ contain a prime, and that the root mean square gap between primes is $\ll (\log{x})^2$. These results therefore show (assuming the Riemann hypothesis) that at least a majority of gaps satisfy bounds similar to those predicted by Cram\'er's conjecture.

Yu \cite{Yu} improved a result of Heath-Brown \cite{HBII} to prove,  assuming the Lindel\"of Hypothesis, that
\begin{equation}
\sum_{p_n\le x}d_n^2\ll x^{1+\epsilon}
\end{equation}
for any $\epsilon>0$. Both of these results show that almost all intervals $[x,x+x^\epsilon]$ contain a prime. Thus a claim `$d_n\ll p_n^\epsilon$' would at least hold for almost all prime gaps.

The best unconditional $L^2$ result thus far is due to Heath-Brown \cite{HBIII}, who proved that
\begin{equation}
\sum_{p_n\le x}d_n^2\ll x^{23/18+\epsilon}.
\end{equation}
This shows that $\sum_{d_n\ge x^{a}}d_n\ll x^{23/18-a+\epsilon}$. It immediately follows that $d_n\ll p_n^{23/36+\epsilon}$ and almost all intervals $[x,x+x^{5/18+\epsilon}]$ contain a prime. Although both of these can be improved with alternative methods, we note that the exponent of $5/18+\epsilon$ is much smaller than the Riemann Hypothesis bound of $1/2+\epsilon$, and so being able to ignore a small number of possible large differences makes the problem much more tractable.

These results should be compared with the lower bound obtained by the Cauchy-Schwartz inequality and the prime number theorem, which gives
\begin{equation}
\sum_{p_n\le x}d_n^2\gg x\log{x}.
\end{equation}

%\newpage
%\input{notation.tex}
\newpage
\section{Main Result}
We aim to improve on Heath-Brown's result \cite{HBIII} and investigate the values of $\nu$ for which we can show
\begin{equation}\label{MainResult}
\sum_{p_n\le x}d_n^2\ll x^{1+\nu+\epsilon}
\end{equation}
for any $\epsilon>0$.

We do this by obtaining $L^2$, $L^4$ and $L^\infty$ bounds on the Chebyschev function $\psi(x)=\sum_{n\le x}\Lambda(n)$ in intervals of size $\tau$.

In particular, we wish to prove:
\begin{thrm}
\[\sum_{p_n\le x}d_n^2\ll x^{5/4+\epsilon}\]
for any $\epsilon>0$.
\end{thrm}
By dyadic subdivision and replacing $\epsilon$ by a finite multiple, we see it is sufficient to prove the following proposition.
\begin{prpstn}\label{mainprop}
For $0\le \tau\le x$ we have
\[\sum_{\substack{4x/\tau \le d_n\le 8x/\tau \\ x \le p_n\le 2x}}d_n^2\ll x^{5/4+10\epsilon}\]
for any $\epsilon>0$
\end{prpstn}
\newpage
\section{Initial Argument}
Proposition \ref{mainprop} holds trivially for $d_n\ll x^{1/4+\epsilon}$ or (by the result of Baker, Harman and Pintz \cite{BakerEtAl}[Theorem 1]) for $d_n\gg x^{21/40}$. Thus we only need to consider
\begin{equation}\label{TauSize}
x^{19/40}\le \tau\le x^{3/4-\epsilon}.
\end{equation}
We follow essentially exactly the same method as Heath-Brown in \cite{HBII} in this section, except that we use Perron's formula to get an estimate for $\psi(x)$ in terms of Dirichlet polynomials instead of zeroes of $\zeta(s)$. (An idea suggested by Heath-Brown in \cite{HBHuxley}). It is the greater control which we get from using this setup which enables us to improve the exponent from $23/18$ to $5/4$.
\subsection{A Combinatorial Identity and Perron's Formula}
We start with the identity:
\begin{equation}
-\frac{\zeta'}{\zeta}(s)=-\frac{\zeta'}{\zeta}(s)(1-M_x(s)\zeta(s))^k+\sum_{j=1}^k(-1)^{j}\binom{k}{j}M_x(s)^j\zeta(s)^{j-1}\zeta'(s),
\end{equation}
where
\begin{align}
&k\in \mathbb{Z}^+ \text{is a positive constant},
&M_x(s)=\sum_{n\le (3x)^{1/k}}\mu(n)n^{-s}.
\end{align}
We will later (equation \eqref{KChoice}) choose $k=60$, since this is sufficient for our purposes.

By our choice of $M_x$ the term $(1-M_x(s)\zeta(s))^k\zeta'(s)/\zeta(s)$ makes no contribution to the coefficient of $n^{-s}$ for $n\le 3x$. Hence equating coefficients of $n^{-s}$ of both sides for $n\le 3x$ gives
\begin{equation}
\Lambda(n)=\sum_{j=1}^{k}(-1)^j\binom{k}{j}K^{(j)}(n)=\sum_{j=1}^kc_jK^{(j)}(n)
\end{equation}
where
\begin{equation}K^{(j)}(n)=\sum_{\substack{\prod_1^{2j}n_i=n\\ n_i\le (3x)^{1/k}\text{ for }i\le j}}\\ \mu(n_1)\dots\mu(n_j)\log{n_{2j}}.\end{equation}
We split $K^{(j)}(n)$ into dyadic intervals for each $n_i$, therefore expressing $\Lambda(n)$ as a linear combination of $O(\log^{2k}{x})$ sums of the form
\begin{equation}J_{N_1,N_2,\dots,N_{2k}}(n)=\sum_{\substack{n_i\in(N_i,2N_i]\forall i\\ \prod_in_i=n}}\mu(n_1)\dots\mu(n_k)\log{n_{2k}}.\end{equation}
We note that
\begin{equation}\label{MuPieceBound}
N_i\le (3x)^{1/k}
\end{equation}
for $i\le k$. We account for the cases when $j<k$ by setting $N_i=1/2$ (and so $n_i=1$) for the `extra' variables.

We now put
\begin{equation}
S_i(s)=
\begin{cases}
\sum_{N_i<n_i\le 2N_i}\mu(n_i)n_i^{-s},&i\le k\\
\sum_{N_i<n_i\le 2N_i}n_i^{-s}&k<i<2k\\
\sum_{N_i<n_i\le 2N_i}(\log{n_i})n_i^{-s}&i=2k
\end{cases}
\end{equation}
and consider the Dirichlet polynomial
\begin{equation}\label{product}
\sum a_nn^{-s}=\sum_{j=1}^kc_j\sum_{\substack{(N_i)_1^{2k}\\N_i\le (3x)^{1/k}\text{ for }i\le k\\ 2^{-2k}x\le\prod_1^{2k}N_i\le 3x\\N_i=1/2\text{ if }j<i\le k\text{ or }k+j\le i<2k}}S_1(s)S_2(s)\dots S_{2k}(s).
\end{equation}
By the above identity, for $x\le n\le 3x$ we have
\begin{equation}a_n=\Lambda(n).\end{equation}
In particular, for $x\le y \le 2x$ and $\tau \ge 2$
\begin{equation}\label{Expansion1}
\psi(y+y/\tau)-\psi(y)=\sum_{y<n\le y+y/\tau}a_n.
\end{equation}
We separate out the case when one of the $N_i>x^{19/20}$, since such very long polynomials require a slightly different treatment.

Thus
\begin{equation}\label{fgSplit}
\sum a_nn^{-s}=\sum f_nn^{-s}+\sum g_nn^{-s}
\end{equation}
where
\begin{equation}\label{LargeCase}
\sum f_nn^{-s}=\sum_{j=1}^kc_j\sum_{\substack{(N_i)_1^{2k}\\N_i\le (3x)^{1/k}\text{ for }i\le k\\2^{-2k}x\le\prod_1^{2k}N_i\le 3x\\ N_i=1/2\text{ if }j<i\le k\text{ or }j+k\le i<2k\\N_i>x^{19/20}\text{ for some }i}}S_1(s)S_2(s)\dots S_{2k}(s),
\end{equation}
and
\begin{equation}\label{NormalCase}
\sum g_nn^{-s}=\sum_{j=1}^kc_j\sum_{\substack{(N_i)_1^{2k}\\N_i\le (3x)^{1/k}\text{ for }i\le k\\2^{-2k}x\le\prod_1^{2k}N_i\le 3x\\N_i=1/2\text{ if }j<i\le k\text{ or }j+k\le i<2k\\N_i\le x^{19/20}\text{ for all }i}}S_1(s)S_2(s)\dots S_{2k}(s).
\end{equation}
We first consider $\sum f_n$. We separate the exceptionally long polynomial $S_{i_0}(s)$, and just consider the remaining product of polynomials as a single polynomial. Since $N_i\le (3x)^{1/k}\le x^{19/20}$ for $i\le k$ we must have $i_0>k$ and so the exceptional polynomial must have all coefficients 1 or $\log{n}$. We will assume that $i_0\ne 2k$, so all the coefficients are identically $1$. The alternative case $i_0=2k$ may be handled similarly.
\begin{equation}\sum_{y<n\le y+y/\tau}f_n=\sum_{j=1}^k\sum_{N_{i_0}>x^{19/20}}\sum_{2^{-2k}x/N_{i_0}\le M\le3x/N_{i_0}}\sum_{\substack{m,n\\y<mn\le y+y/\tau\\ M<m\le 2^{2k-1}M\\N_{i_0}<n\le 2N_{i_0}}}b_m^{(j)}\end{equation}
for some coefficients $b_m^{(j)}\ll x^\epsilon$.

We let
\begin{equation}
\mathcal{B}_\tau=\left\{z:x\le z\le 2x, mN_{i_0}>z>\frac{mN_{i_0}}{1+1/\tau}\text{ for some }m, N_{i_0}\right\}
\end{equation}
and consider separately $y\notin \mathcal{B}_\tau$ and $y\in \mathcal{B}_\tau$. We note that 
\begin{equation}\text{meas}(\mathcal{B}_\tau)\ll \sum_{m,N_{i_0}} \frac{mN_{i_0}}{\tau}\ll\sum_{m,N_{i_0}} \frac{x}{\tau}\ll \frac{x^{21/20+\epsilon}}{\tau}.\label{Bsize}
\end{equation}
Thus in particular $[x,2x]-\mathcal{B_\tau}\ne \emptyset$ and $\mathcal{B_\tau}$ represents only a small subset of $y$ with $x\le y\le 2x$.

If $y\notin \mathcal{B}_\tau$ then for any $m$ we have
\begin{equation}
\#\{n:\frac{y}{m}<n\le \frac{y(1+1/\tau)}{m}, N_{i_0}<n\le 2N_{i_0}\}=
\begin{cases}
\frac{y}{\tau m}+O(1),\quad &mN_{i_0}<y\le 2mN_{i_0}\\
0,&\text{otherwise}
\end{cases}
\end{equation}
Thus the sum over $n$ is over $y/(\tau m)+O(1)$ terms if $mN_{i_0}<y<2mN_{i_0}$ or is empty. Hence
\begin{align}
\sum_{y<n\le y+y/\tau}f_n&=\textbf{1}_{\mathcal{B}^C_\tau}(y)\sum_{\substack{m,j,M,N_{i_0}\\ mN_{i_0}<y<2mN_{i_0}}}\left(\frac{b_m^{(j)}y}{\tau m}+O(x^\epsilon)\right)+\textbf{1}_{\mathcal{B}_\tau}(y)\sum_{y<n\le y+y/\tau}f_n\nonumber\\
&=\textbf{1}_{\mathcal{B}^C_\tau}(y)\frac{A_1(y)}{\tau}+O\left(\sum_M Mx^{\epsilon}\right)+\textbf{1}_{\mathcal{B}_\tau}(y)\left(\sum_{y\le n\le y+y/\tau}f_n\right),
\end{align}
where we have defined
\begin{equation}A_1(y)=\sum_{\substack{m,j,M,N_{i_0}\\mN_{i_0}<y<2mN_{i_0}}}\frac{b_m^{(j)}y}{m}.\end{equation}
We note that $A_1(y)$ is independent of $\tau$. Since $M\ll x/N_{i_0}\ll x^{1/20}$ we have
\begin{align}
\sum_{y<n\le y+y/\tau}f_n&=\frac{A_1(y)}{\tau}+O(x^{1/20+2\epsilon})+\textbf{1}_{\mathcal{B}_\tau}(y)O\left(\sum_{y\le n\le y+y/\tau}f_n+A_1(y)/\tau\right)\nonumber\\
&=\frac{A_1(y)}{\tau}+E_1+E_2.\label{fpart}
\end{align}
where $E_1=O(x^{1/19})$ and $E_2=0$ when $y\notin \mathcal{B}_\tau(y)$.

This gives us a `main term' $A_1(y)/\tau$, which we will estimate in Lemma \ref{MainEstimate}, and two error terms $E_1$ and $E_2$. $E_1$ is always small, and so causes no problems. $E_2$ can only be large when $y\in\mathcal{B_\tau}$, which is a suitably small set to cause us no problems.

We now consider $\sum g_n n^{-s}$. To ease notation we put
\begin{equation}S(s)=\prod_{i=1}^{2k}S_i(s),\end{equation}
and we let the unlabelled sum $\displaystyle\sum$ represent the sum
\begin{equation}\sum_{j=1}^kc_j\sum_{\substack{(N_i)_1^{2k}\\N_i\le (3x)^{1/k}\text{ for }i\le k\\2^{-2k}x\le\prod_1^{2k}N_i\le 3x\\N_i=1/2\text{ if }j<i\le k\text{ or }j+k\le i<2k\\N_i\le x^{19/20}\text{ for all }i}}\end{equation}
which appears in the right hand side of \eqref{NormalCase}.

Perron's formula states that for $T>2$, $x>0$, $x\ne 1$ and $1<\sigma\le 2$ we have
\begin{equation}\frac{1}{2\pi i}\int_{\sigma-iT}^{\sigma+iT}\frac{x^s}{s}ds=H(x)+O\left(\frac{x^\sigma}{T|\log{x}|}\right),\end{equation}
where $H(x)=0$ for $x<1$ and $H(x)=1$ for $x>1$.

Using Perron's formula and putting $c=1+1/\log{y}$:
\begin{align}
\sum_{y<n\le y+y/\tau}g_n&=\frac{1}{2\pi i}\int_{c-iT_0}^{c+iT_0}\frac{y^s}{s}\left(\left(1+\frac{1}{\tau}\right)^s-1\right)\left(\sum S(s)\right)ds+E_3 \nonumber\\ 
&=\frac{1}{2\pi i}\int_{c-iT_1}^{c+iT_1}\frac{y^s}{s}\left(\left(1+\frac{1}{\tau}\right)^s-1\right)\left(\sum S(s)\right)ds+E_3+E_4 \nonumber\\ 
&=\frac{1}{2\pi i}\int_{c-iT_1}^{c+iT_1}\frac{y^s}{\tau}\left(\sum S(s)\right)ds+E_3+E_4+E_5 \nonumber\\ 
&=\frac{A_2(y)}{\tau}+E_3+E_4+E_5. \label{errors}
\end{align}
Here
\begin{align}
A_2(y)&=\frac{1}{2\pi i}\int_{c-iT_1}^{c+iT_1}y^s\left(\sum S(s)\right)ds,\\
E_3&=E_3(y,\tau)=O\left(\frac{y\log^2{y}}{T_0}+\log{y}\right),\label{E1}\\
E_4&=E_4(y,\tau)=O\left(\left|\int_{c+iT_1}^{c+iT_0} y^sC_1(s)\left(\sum S(s)\right)ds\right|\right),\label{E2}\\
E_5&=E_5(y,\tau)=O\left(\left|\int_{c-iT_1}^{c+iT_1} y^sC_2(s)\left(\sum S(s)\right)ds\right|\right),\label{E3}\\
C_1(s)&=\frac{1}{s}\left(\left(1+\frac{1}{\tau}\right)^s-1\right),\\
C_2(s)&=\frac{1}{s}\left(\left(1+\frac{1}{\tau}\right)^s-1-\frac{s}{\tau}\right).
\end{align}
We note that
\begin{align}
A_2(y)&\text{ is independent of }\tau,\label{A2Indep}\\
C_1(s)&\ll \frac{1}{\tau},\label{C1Bound}\\
C_2(s)&\ll \frac{|s|}{\tau^2}.\label{C2Bound}
\end{align}
Therefore we have a `main term' $A_2(y)/\tau$ and error terms $E_3,E_4$ and $E_5$. We will show that $E_3$ and $E_5$ are small, and so do not cause any problems in Lemma \ref{MainEstimate} below. If we can show that $E_4$ is only large on a small set, then we will have a suitably accurate estimate of $\psi(y+y/\tau)-\psi(y)$.

Putting together \eqref{fpart} and \eqref{errors}, using \eqref{Expansion1} and \eqref{fgSplit}, and setting $A(y)=A_1(y)+A_2(y)$ we get
\begin{align}
\psi(y+y/\tau)-\psi(y)&=\frac{A(y)}{\tau}+E_1+E_2+E_3+E_4+E_5.
\end{align}
\begin{lmm}\label{MainEstimate}
For $T_0=\tau(\log{y})^3$, $T_1=y^{1/8}$ and $x^{1/3}\le \tau \le x^{3/4}$ we have
\begin{align*}
(i)\qquad&E_1, E_3, E_5= o\left(\frac{y}{\tau}\right)
(ii)\qquad&A(y)\sim y\text{ for }y\notin\mathcal{B}_{y^{1/3}}.
\end{align*}
\end{lmm}
\begin{proof}
(i): Estimate of $E_1$, $E_3$, $E_5$.

Since $\tau\le x^{3/4}$ and $x\le y$, we have
\begin{equation}E_1\ll x^{1/19}=o\left(\frac{y}{\tau}\right).\end{equation}
Since $\tau\le y^{3/4}$ and $T_0=\tau(\log{y})^3$, we have
\begin{equation}E_3=O\left(\frac{y(\log{y})^2}{T_0}\right)=O\left(\frac{y(\log{y})^2}{\tau(\log{y})^3}+\log{y}\right)=o\left(\frac{y}{\tau}\right).\end{equation}
We have
\begin{equation}|S_i(c+it)|\le \sum_{N_i<n_i\le 2N_i}(\log{n_i})n_i^{-c}\ll \log{y}\end{equation}
for all $i$. Since $S$ is a product of the $S_i$ we have
\begin{equation}\left|S(c+it)\right|\ll y^{\epsilon}.\end{equation}
Thus, since $T_1=y^{1/8}$ and $\tau\ge y^{1/3}$ and $C_2(s)\ll |s|\tau^{-2}$ (by \eqref{C2Bound}) and $\sum$ is a sum over $\ll y^\epsilon$ terms, we have that
\begin{align}
E_5&\ll \int_{c-iT_1}^{c+iT_1}y^c\left|C_2(s)\right|\left|\sum S(s)\right||ds|\nonumber\\
&\ll \frac{y^{c+2\epsilon}T_1^2}{\tau^2}\nonumber\\
&\ll \frac{y^{11/12+2\epsilon}}{\tau}\nonumber\\
&=o\left(\frac{y}{\tau}\right).
\end{align}
(ii): Estimate of $A(y)$.

The idea for estimating $A(y)$ is as follows. For $\tau=\tau_0$ with $\tau_0$ `small' we have that $\psi(y+y/\tau)-\psi(y)\sim y/\tau$. We have that $E_1,E_3,E_5$ are all small relative to this, and that $E_2$ is zero outside $\mathcal{B}_\tau$. Therefore, provided we can show $E_4$ is small for this value of $\tau$, we can bound $A(y)$ from below when $y$ is not in $\mathcal{B}_{\tau_0}$ (which covers almost all values of $y$). Since $A(y)$ is independent of $\tau$, this bound holds for any size of $\tau$, giving the result. We proceed to make this precise.

Huxley's Theorem \cite{Huxley} states that
\begin{equation}\psi(a+b)-\psi(a)\sim b,\end{equation}
for $b>a^{7/12+\epsilon}$.

Using Huxley's Theorem taking $a=y$, $b=y/\tau$ and $\tau= y^{1/3}$ we obtain
\begin{equation}\psi(y+y/\tau)-\psi(y)\sim y\tau^{-1}.\end{equation}
For this value of $\tau$ we still have
\begin{equation}E_1,E_3,E_5=o(y\tau^{-1}).\end{equation}
By Heath-Brown \cite{HBHuxley}[Lemma 3] we have
\begin{equation}\int_T^{2T}|S_1(1/2+it)\dots S_{2k}(1/2+it)|dt\ll x^{1/2}(\log{x})^{-12}\end{equation}
uniformly for $\exp((\log{x})^{1/3})\le T\le x^{5/12-\epsilon}$.

A precisely analogous argument yields
\begin{equation}\label{E4Small}
\int_T^{2T}|S_1(c+it)\dots S_{2k}(c+it)|dt\ll x^{1-c}(\log{x})^{-A}
\end{equation}
for any constant $A>0$ and uniformly for $\exp((\log{x})^{1/3})\le T\le x^{5/12-\epsilon}$. We choose $A=2k+2$ ($=122$) since this will be sufficient for our purposes.

When $\tau=y^{1/3}$ we have $\exp((\log{x})^{1/3})\le y^{1/8}=T_1$ and $T_0=\tau(\log{y})^3\le x^{5/12-\epsilon}$. We can therefore use \eqref{E4Small} uniformly for $T\in[T_1,T_0]$.

Thus, since $|C_1(s)|\ll \tau^{-1}$ (by \eqref{C1Bound}), we have
\begin{align}
E_4&\ll\left| \int_{c+iT_1}^{c+iT_0}y^{s}C_1(s)\sum S(s)ds\right|\nonumber\\
&\ll \frac{y^{c}}{\tau}\sum\int_{c+iT_1}^{c+iT_0}\left|S(s)\right|ds\nonumber\\
&\ll \frac{y^{c}(\log{y})}{\tau}\sum\sup_{T\in[T_0,T_1]}\int_T^{2T}\left|S_1(c+it)\dots S_{2k}(c+it)\right|dt\nonumber\\
&\ll \frac{y}{(\log{y})^{2k+1}\tau}\sum 1.
\end{align}
Since the sum is over $O((\log{y})^{2k})$ terms, this gives
\begin{equation}E_4=o\left(\frac{y}{\tau}\right).\end{equation}
Thus for $\tau=y^{1/3}$ we have
\begin{equation}E_1, E_3, E_4,E_5=o(y\tau^{-1}),\qquad\psi(y+y\tau^{-1})-\psi(y)\sim y\tau^{-1}.\end{equation}
Moreover, $E_2=0$ for $y\notin \mathcal{B}_{y^{1/3}}$ when $\tau=y^{1/3}$.

Hence for $\tau=y^{1/3}$ and $y\notin \mathcal{B}_{y^{1/3}}$
\begin{equation}A(y)\sim y.\end{equation}
Since $A(y)$ is independent of $\tau$, this must hold for all values of $\tau$.
\end{proof}
Thus
\begin{equation}\psi(y+y/\tau)-\psi(y)-\frac{y}{\tau}=E_2+E_4+\textbf{1}_{\mathcal{B}_{y^{1/3}}}(y)O\left(\frac{A(y)}{\tau}\right)+o\left(\frac{y}{\tau}\right).\end{equation}
We let $E_6=E_2+\textbf{1}_{\mathcal{B}_{y^{1/3}}}(y)A(y)/\tau$. Since $\mathcal{B}_{y^{1/3}}\supset\mathcal{B}_\tau$ for $\tau\ge y^{1/3}$ we see that $E_6=0$ if $y\notin\mathcal{B}_{y^{1/3}}$. Therefore
\begin{equation}
\psi(y+y/\tau)-\psi(y)-\frac{y}{\tau}=E_4+E_6+o\left(\frac{y}{\tau}\right),\label{KeyPoint1}
\end{equation}
where $E_6=0$ if $y\notin \mathcal{B}_{y^{1/3}}$.

By definition of $\psi$, we also have
\begin{equation}\psi(y+y/\tau)-\psi(y)=\sum_{\substack{k,p \text{ prime}\\y\le p^k\le y+y/\tau}}\log{p}.\end{equation}
The key point is that if there are no primes in the interval $[y,y+y/\tau]$ then there are no terms with $k=1$. Hence
\begin{align}
\psi(y+y/\tau)-\psi(y)&\le \sum_{2\le k\le \log{y}}\sum_{y^{1/k}\le p\le y^{1/k}+(y/\tau)^{1/k}}\log{y}\nonumber\\
&\ll(\log{y})^2(y/\tau)^{1/2}\nonumber\\
&=o\left(\frac{y}{\tau}\right).\label{NoZeroes}
\end{align}
Thus if there are no primes in the interval $[y,y+y/\tau]$ then the left hand side of \eqref{KeyPoint1} is $\gg y\tau^{-1}$. This means that $E_4+E_6\gg y\tau^{-1}$. The term $E_6$ is only non-zero on $\mathcal{B}_{y^{1/3}}$, which is a small set, and so cannot be large frequently. Moreover, $E_4$ can only be large when $\sum S(c+it)$ is large, and we can show this does not happen too often by estimates on the frequency with which Dirichlet Polynomials can take large values. Thus we can show that the interval $[y,y+y/\tau]$ rarely contains no primes.

We split the sum $\sum S_1\dots S_{2k}$ up into subsums dependent on the size of each of the $S_i$, to show that $E_4$ cannot be large often.

We put
\begin{equation}
\mathcal{S}=\mathcal{S}(\sigma_1,\dots,\sigma_{2k}):=\left\{m\in\mathbb{Z}:N_i^{-c+\sigma_i} \le \sup_{t\in[m,m+1]}|S_i|\le 2N_i^{-c+\sigma_i}\forall i\right\}
\end{equation}
for each $\sigma_i\in\left\{1, 1-\frac{\log{2}}{\log{N_i}},1-\frac{2\log{2}}{\log{N_i}}, \dots,-\frac{\log{x}}{\log{N_i}}\right\}$. We let $\mathcal{S}_0$ cover the remaining values of $m$, so $\mathcal{S}_0=\{m\in\mathbb{Z}:\sup_{t\in[m,m+1]}|S_i|\le N_i^{-c}x^{-1}$ for some $i\}$.

We let $\sum_{(\sigma_i)}$ represent the sum over all the $O((\log{x})^{2k})$ values of $(\sigma_i)_1^{2k}$.

Hence splitting $E_4$ into terms corresponding to the choices of $(\sigma_i)$ we get
\begin{align}
E_4&\ll \left|\int_{c+iT_1}^{c+iT_0}y^sC_1(s)\left(\sum S(s)\right)ds\right|\nonumber\\
&\ll \sum_{(\sigma_i)}\sum\left|\sum_{m\in \mathcal{S}\cap[T_0,T_1]}\int_{c+im}^{c+i(m+1)}y^sC_1(s) S(s)ds\right|\nonumber\\
&\qquad+\sum\sum_{m\in\mathcal{S}_0\cap[T_1,T_0]}\int_{c+im}^{c+i(m+1)}\left|y^sC_1(s) S(s)\right||ds|.
\end{align}
The sum $\sum$ is over $O((\log{x})^{2k})$ terms, $|C_1(s)|\ll \tau^{-1}$ (by \eqref{C1Bound}) and for $m\in \mathcal{S}_0$ we have $|S(s)|\le x^{-1}$. Therefore the last term is
\begin{equation}\ll (\log{x})^{2k}T_0x\tau^{-1}x^{-1}=o\left(\frac{x}{\tau}\right).\end{equation}
We split the range of integration of the first term into $O(\log(1+T_0/T_1))$ dyadic intervals. This gives
\begin{align}
E_4&\ll \log(1+T_0/T_1)\sup_{T\in[T_1,T_0]}\sum_{(\sigma_i)}\sum\left|\sum_{m\in \mathcal{S}\cap[T,2T]}\int_{c+im}^{c+i(m+1)}y^sC_1(s) S(s)ds\right|\nonumber\\
&\qquad\qquad+\sum_{(\sigma_i)}\sum\int_{T_1/2}^{T_1}y^c\tau^{-1}|S(c+it)|dt+o\left(\frac{x}{\tau}\right).\label{E4Split}
\end{align}
We put
\begin{equation}E((N_i),(\sigma_i);y,T)=\left|\sum_{m\in\mathcal{S}\cap[T,2T]}\int_{m}^{m+1}y^{c+it}C_1(c+it)S(c+it)dt\right|.\end{equation}
We note that the sum $\sum_{(\sigma_i)}\sum$ is a sum over $O((\log{x})^{4k})$ terms. Thus the first term on the right hand side of \eqref{E4Split} is
\begin{equation}\ll (\log{x})^{4k+1}\sup_{(N_i),(\sigma_i),T\in[T_1,T_0]}E((N_i),(\sigma_i);y,T)\end{equation}
where $(\sigma_i)_1^{2k}$ and $(N_i)_1^{2k}$ are constrained by
\begin{align}
(i)&: \sigma_i\le 1\quad \forall i,\label{C1}\\
(ii)&: x\ll \prod_{i=1}^{2k} N_i\ll x,\label{C2}\\
(iii)&: N_i\le (3x)^{1/k}\text{ if }i\le k,\label{C3}\\
(iv)&: N_i\le x^{19/20}\quad \forall i. \label{C4}
\end{align}
Since $T_1=y^{1/8}$ we can use \eqref{E4Small} with $A=4k+1$ to bound the second term on the right hand side of \eqref{E4Split}. This gives
\begin{align}
\sum_{(\sigma_i)}\sum\int_{T_1/2}^{T_1}y^c\tau^{-1}|S(c+it)|dt&\ll (\log{x})^{4k}y^c\tau^{-1}y^{1-c}(\log{x})^{-4k-1}\nonumber\\
&=o\left(\frac{y}{\tau}\right).
\end{align}
Since $\tau\le x^{3/4}$ we can bound the third term on the right hand side of \eqref{E4Split} trivially.
\begin{align}
(\log{x})^{2k}T_0x\tau^{-1}x^{-1}&\ll (\log{x})^{2k+3}\nonumber\\
&=o\left(\frac{y}{\tau}\right).
\end{align}
Putting this together, we obtain:
\begin{align}
&\left|\psi\left(y+y/\tau\right)-\psi(y)-\frac{y}{\tau}\right|\nonumber\\
&\qquad\qquad\ll (\log{x})^{4k+1}\sup_{(N_i),(\sigma_i),T\in[T_1,T_0]}E((N_i),(\sigma_i);y,T)+E_6+o\left(\frac{x}{\tau}\right).\label{setup}\end{align}
We now want to show that there cannot be many large gaps between primes by showing that the $L^2$, $L^4$ and $L^\infty$ norms of $E((N_i),(\sigma_i);y,T)$ cannot all be simultaneously large.
\newpage
\subsection{The Basic Lemma}
\begin{lmm}\label{BasicLemma}
We have
\begin{equation}\left|\psi(y+y/\tau)-\psi(y)-y/\tau\right|\ll (\log{x})^{4k+1}\sup_{(N_i),(\sigma_i),T\in[T_1,T_0]}E((N_i),(\sigma_i);y,T)+E_6+o\left(\frac{y}{\tau}\right)\end{equation}
where the supremum is constrained by \eqref{C1},\eqref{C2},\eqref{C3},\eqref{C4} and $E((N_i),(\sigma_i);y,T)$ satisfies, for any $\epsilon>0$:
\begin{align}
E((N_i),(\sigma_i);y,T)&\ll x_1^{\sigma}\tau^{-1}R(T),\label{Einfty}\\
\int_x^{2x}\left|E((N_i),(\sigma_i);y,T)\right|^2dy &\ll x_1^{1+2\sigma+\epsilon}\tau^{-2}R(T),\label{EL2}\\
\int_x^{2x}\left|E((N_i),(\sigma_i);y,T)\right|^4dy&\ll x_1^{1+4\sigma+\epsilon}\tau^{-4}R^*(T).\label{EL4}
\end{align}
Here
\begin{align}
\mathcal{S}&=\mathcal{S}(\sigma_1,\dots,\sigma_{2k})=\left\{m:N^{-c+\sigma_i} \le \sup_{t\in[m,m+1]}|S_i|\le 2N^{-c+\sigma_i}\forall i\right\},\\
\mathcal{S^*}&=\mathcal{S^*}(\sigma_1,\dots,\sigma_{2k})=\left\{(m_1,m_2,m_3,m_4)\in\mathcal{S}^4:m_1+m_2=m_3+m_4\right\},\\
R(T)&=R(T,\sigma_1,\dots,\sigma_{2k})=\#\left(\mathcal{S}\cap[T,2T]\right),\\
R^*(T)&=R^*(T,\sigma_1,\dots,\sigma_{2k})=\#\left(\mathcal{S^*}\cap[T,2T]^4\right),\\
x_1&=\prod_{i=1}^{2k}N_i,\\
x_1^\sigma&=\prod_iN_i^{\sigma_i}.
\end{align}
\end{lmm}
\begin{proof}
The Proof follows exactly the same lines as that of Heath-Brown in \cite{HBI} and \cite{HBII} but using Dirichlet polynomials instead of zeroes of $\zeta(s)$.

We note that
\begin{equation}x\ll x_1=\prod_{i=1}^{2k}N_i\ll x,\qquad x\ll y\ll x.\end{equation}
We will find it slightly more convenient to work with $x_1$ rather than $x$ in our later arguments, and so we introduce it now.

We recall that $|C_1(s)|\ll \tau^{-1}$ (by \eqref{C1Bound}) and that $S(c+it)\ll \prod_{i=1}^{2k} N_i^{c-\sigma_i}=x_1^{\sigma-c}$ for $t\in[m,m+1]$ and $m\in \mathcal{S}$.

(i): $L^\infty$ estimate.
\begin{align}
E((N_i),(\sigma_i);y,T)&=\left|\sum_{m\in\mathcal{S}\cap[T,2T]}\int_{m}^{m+1}y^{c+it}C_1(c+it)S(c+it)dt\right|\nonumber\\
&\ll\sum_{m\in\mathcal{S}\cap[T,2T]}\int_{m}^{m+1}y^c\tau^{-1}x_1^{-c+\sigma}dt\nonumber\\
&\ll x_1^\sigma\tau^{-1}\sum_{m\in\mathcal{S}\cap[T,2T]}1\nonumber\\
&\ll x_1^{\sigma}\tau^{-1}R(T).
\end{align}
(ii): $L^2$ estimate.
\begin{align}
\int_x^{2x}|E&(N_i),(\sigma_i);y,T|^2dy\nonumber\\
&=\sum_{m_1,m_2\in\mathcal{S}\cap[T,2T]}\int_{c+im_1}^{c+i(m_1+1)}\int_{c+im_2}^{c+i(m_2+1)}\nonumber\\
&\qquad\qquad\qquad\qquad\qquad\left(\int_x^{2x}y^{s_1+\overline{s}_2}dy\right)C_1(s_1)S(s_1)\overline{C_1(s_2)}\overline{S(s_2)}ds_1ds_2\nonumber\\
&\ll x_1^{3}\sum_{m_1,m_2\in\mathcal{S}\cap[T,2T]}\int_{c+im_1}^{c+i(m_1+1)}\int_{c+im_2}^{c+i(m_2+1)}\frac{|C_1(s_1)S(s_1)C_1(s_2)S(s_2)|}{|1+s_1+\overline{s}_2|}ds_1ds_2\nonumber\\
&\ll x_1^{2\sigma+1}\tau^{-2}\sum_{m_1,m_2\in\mathcal{S}\cap[T,2T]}\int_{m_1}^{m_1+1}\int_{m_2}^{m_2+1}\frac{1}{1+|t_1-t_2|}dt_1dt_2\nonumber\\
&\ll \frac{x_1^{2\sigma+1}\log{x_1}}{\tau^2}\sum_{m_1\in\mathcal{S}\cap[T,2T]}1\nonumber\\
&\ll x_1^{1+2\sigma+\epsilon}\tau^{-2}R(T).
\end{align}
(iii): $L^4$ estimate.
\begin{align}
\int_x^{2x}|E((N_i)&,(\sigma_i);y,T)|^4dy\nonumber\\
&\ll \int_x^{2x}\left|\sum_{m\in \mathcal{S}\cap[T,2T]}\int_m^{m+1}y^{c+it}C_1(c+it)S(c+it)dt\right|^4dy\nonumber\\
&\ll \sum_{m_1,m_2,m_3,m_4\in \mathcal{S}\cap[T,2T]}\int_{m_1}^{m_1+1}\int_{m_2}^{m_2+1}\int_{m_3}^{m_3+1}\int_{m_4}^{m_4+1}\nonumber\\
&\qquad\qquad \left|\int_x^{2x}y^{4c+i(t_1+t_2-t_3-t_4)}dy\right|\prod_{j=1}^{4}\left(\left|C_1(c+it_j)S(c+it_j)\right|dt_j\right)\nonumber\\
&\ll x_1^{5}\left(\tau^{-4}\left(\prod_iN_i^{-4c+4\sigma_i}\right)\sum_{m_j\in\mathcal{S}\cap[T,2T]}\frac{1}{1+|m_1+m_2-m_3-m_4|}\right)\nonumber\\
&\ll \frac{x_1^{1+4\sigma}}{\tau^4}\sum_{m_j\in\mathcal{S}\cap[T,2T]}\frac{1}{1+|m_1+m_2-m_3-m_4|}.
\end{align}
We wish to bound the inner sum. We let
\begin{equation}g(v):=\#\{(m_1,m_2,m_3,m_4)\in(\mathcal{S}\cap[T,2T])^4:m_1+m_2-m_3-m_4=v\},\end{equation}
\begin{equation}\mathcal{S}^*=\{(m_1,m_2m_3,m_4)\in\mathcal{S}^4:m_1+m_2-m_3-m_4=0\}.\end{equation}
Then
\begin{equation}\sum\frac{1}{1+|m_1+m_2-m_3-m_4|}\ll \sum_{|v|\le 4T}\frac{g(v)}{1+|v|}.\end{equation}
But we have
\begin{align}
g(v)&=\int_0^1\left|\sum_{m\in \mathcal{S}\cap[T,2T]} e(mu)\right|^4e(-vu)du\nonumber\\
&\le\int_0^1\left|\sum_{m\in \mathcal{S}\cap[T,2T]} e(mu)\right|^4du=g(0).
\end{align}
Hence
\begin{equation}\sum_{m_j\in\mathcal{S}\cap[T,2T]}\frac{1}{1+|m_1+m_2-m_3-m_4|}\ll \#\left(\mathcal{S}^*\cap[T,2T]\right)\log{T}.\end{equation}
This gives us
\begin{equation}\int_x^{2x}|E((N_i),(\sigma_i);y,T)|^4dy \ll x_1^{1+4\sigma+\epsilon}\tau^{-4}R^*(T).\end{equation}
\end{proof}
%\newpage
\subsection{Estimation of $\sum d_n^2$}
We now use Lemma \ref{BasicLemma} to estimate $\sum d_n^2$.

Suppose $p_{n+1}-p_n\ge 4x/\tau$ and $x\le p_n\le 2x$. Let
\begin{equation}
y\in(p_n,(p_{n+1}+p_n)/2)\label{ybounds}
\end{equation}
so that, for $x\le y\le 2x$, we have
\begin{equation}p_n<y<y+y/\tau\le p_{n+1}.\end{equation}
Hence there are no primes in the interval $(y,y+y/\tau)$. In this case, by \eqref{NoZeroes} we have
\begin{equation}\psi(y+y/\tau)-\psi(y)=o\left(\frac{y}{\tau}\right).\end{equation}
Thus Lemma \ref{BasicLemma} yields
\begin{equation}
\sup_{(N_i),(\sigma_i),T}E((N_i),(\sigma_i);y,T)+E_6\gg \frac{x}{\tau(\log{x})^{4k+1}}.\label{Elarge1}
\end{equation}
We now wish to show that this cannot be the case too frequently.

Since $E_6=0$ for $y\notin \mathcal{B}_{y^{1/3}}$, we see that
\begin{equation}E_6\gg \frac{x}{\tau(\log{x})^{4k+1}}\end{equation}
can only hold on a set of measure at most
\begin{equation}
\text{meas}(\mathcal{B}_{y^{1/3}})\ll x^{43/60+\epsilon}\label{E6Measure}
\end{equation}
by \eqref{Bsize} and \eqref{TauSize}.

Suppose that for some choice of $(N_i),(\sigma_i),T$ we have
\begin{equation}\label{Elarge}
E((N_i),(\sigma_i);y,T)\gg \frac{x}{\tau(\log{x})^{4k+1}}.
\end{equation}
By \eqref{Einfty} we must have
\begin{equation}
R(T)\gg x_1^{1-\sigma}(\log{x_1})^{-4k-1}.
\end{equation}
We now wish to estimate how frequently \eqref{Elarge} can occur. By \eqref{Bsize} and \eqref{EL2} we see that \eqref{Elarge} can hold on a set of measure
\begin{align}
&\ll x_1^{2\sigma-1+\epsilon}R(T).
\end{align}
Similarly from \eqref{EL4} we see that this set has measure
\begin{equation}\ll x_1^{4\sigma-3+\epsilon}R^*(T).\end{equation}
Therefore \eqref{Elarge} holds on a set of measure
\begin{equation}\ll\begin{cases}
\min\left(x_1^{2\sigma-1+\epsilon}R(T),x_1^{4\sigma-3+\epsilon}R^*(T)\right),\qquad &R(T)\gg x_1^{1-\sigma}(\log{x_1})^{-4k-1}\\
0,&\text{otherwise}.
\end{cases}\end{equation}
There are $O(x_1^\epsilon)$ choices for $(N_i)_1^{2k}$, $(\sigma_1)_1^{2k}$ and $T$. Therefore
\begin{equation}\sup_{(N_i),(\sigma_i),T}E((N_i),(\sigma_i);y,T)\gg \frac{x}{\tau(\log{x})^{4k+1}}\end{equation}
can only hold on a set of cardinality
\begin{equation}
\ll x_1^\epsilon\sup_{\substack{(N_i),(\sigma_i),T\\T\in[T_1,T_0]\\R(T)\gg x_1^{1-\sigma}(\log{x_1})^{-4k-1}}}\left(x_1^{\epsilon}\min\left(x_1^{2\sigma-1}R(T),x_1^{4\sigma-3}R^*(T)\right)\right). \label{E4Measure}
\end{equation}
Putting \eqref{E6Measure} and \eqref{E4Measure} together, we see that the set of $y$ such that $y\in (p_n,p_n/2+p_{n+1}/2)$ with $p_n{+1}-p_n\ge 4x/\tau$ and $x\le p_n\le 2x$ must have measure
\begin{equation}\ll \sup_{\substack{(N_i),(\sigma_i),T\\T\in[T_1,T_0]\\R(T)\gg x_1^{1-\sigma}(\log{x_1})^{-4k-1}}}\left(x_1^{2\epsilon}\min\left(x_1^{2\sigma-1}R(T),x_1^{4\sigma-3}R^*(T)\right)\right)+x_1^{43/60+\epsilon}.\end{equation}
However, this set trivially has measure
\begin{equation}\ge \sum_{\substack{p_{n+1}-p_n\ge 4x/\tau\\p_n \ge x\\(p_n+p_{n+1})/2\le 2x}}\frac{p_{n+1}-p_n}{2}.\end{equation}
Therefore we have
\begin{align}
\sum_{\substack{p_{n+1}-p_n\ge 4x/\tau\\p_n \ge x\\(p_n+p_{n+1})/2\le 2x}}\frac{p_{n+1}-p_n}{2}&\ll \sup_{\substack{(N_i),(\sigma_i),T\\T\in[T_1,T_0]\\R(T)\gg x_1^{1-\sigma}(\log{x_1})^{-4k-1}}}\left(x_1^{2\epsilon}\min\left(x_1^{2\sigma-1}R(T),x_1^{4\sigma-3}R^*(T)\right)\right)\nonumber\\
&\qquad\qquad\qquad\qquad\qquad+x_1^{43/60+\epsilon}.
\end{align}
There is at most one prime $p_n$ such that $p_n\le 2x < (p_n+p_{n+1})/2$. Hence
\begin{align}
\sum_{\substack{4x/\tau\le p_{n+1}-p_n\le 8x/\tau\\x \le p_n \le 2x}}(p_{n+1}-p_n)&\ll 
\sup_{\substack{(N_i),(\sigma_i),T\\T\in[T_1,T_0]\\R(T)\gg x_1^{1-\sigma}(\log{x_1})^{-4k-1}}}\left(x_1^{2\epsilon}\min\left(x_1^{2\sigma-1}R(T),x_1^{4\sigma-3}R^*(T)\right)\right)\nonumber\\
&\qquad\qquad\qquad\qquad+x_1^{43/60+\epsilon}+\frac{x_1}{\tau}.
\end{align}
Thus
\begin{align}
\sum_{\substack{4x/\tau\le p_{n+1}-p_n\le 8x/\tau\\x \le p_n \le 2x}}(p_{n+1}-p_n)^2&\ll \sup_{\substack{(N_i),(\sigma_i),T\\T\in[T_1,T_0]\\R(T)\gg x_1^{1-\sigma}(\log{x_1})^{-4k-1}}}\left(x_1^{2\epsilon}\tau^{-1}\min\left(x_1^{2\sigma}R(T),x_1^{4\sigma-2}R^*(T)\right)\right)\nonumber\\
&\qquad\qquad\qquad\qquad+\frac{x_1^{103/60+\epsilon}}{\tau}+\frac{x_1^2}{\tau^2}.
\end{align}
This is good enough to prove that
\begin{equation}\sum_{\substack{4x/\tau\le d_n\le 8x/\tau\\x\le p_n\le 2x}}d_n^2\ll x^{1+\nu+10\epsilon}\end{equation}
 if we can prove that
\begin{equation}
\sup_{\substack{(N_i),(\sigma_i),T\\T\in[T_1,T_0]\\R(T)\gg x_1^{1-\sigma}(\log{x_1})^{-4k-1}}}\left(x_1^{2\epsilon}\min\left(x_1^{2\sigma}\tau^{-1}R(T),x_1^{4\sigma-2}\tau^{-1}R^*(T)\right)\right)+x_1^{103/60+\epsilon}/\tau\ll x_1^{1+\nu+10\epsilon}.
\end{equation}
Therefore (recalling $T_0=\tau(\log{x})^3$) we have proven the following proposition.
\begin{prpstn}\label{PropMainResult}
Let $x^{19/40}\le\tau\le x^{3/4-\epsilon}$ and $\nu\ge 29/120$.

If, uniformly for all $T\in[T_1,T_0]$ and for all possible $(N_i),(\sigma_i)$ satisfying \eqref{C1}, \eqref{C2}, \eqref{C3} and \eqref{C4}, at least one of the following holds:
\begin{align}
(i):&R(T)\ll x_1^{1-\sigma}(\log{x_1})^{-4k-2},\\
(ii):&R(T)\ll T_0x_1^{1+\nu-2\sigma+8\epsilon},\\
(iii):&R^*(T)\ll T_0x_1^{3+\nu-4\sigma+8\epsilon}\label{A2},
\end{align}
then we have
\begin{equation}\sum_{\substack{4x/\tau\le d_n\le 8x/\tau\\x\le p_n\le 2x}}d_n^2\ll x^{1+\nu+10\epsilon}.\end{equation}
\end{prpstn}
\newpage
\section{Large Values of Dirichlet Polynomials}
We recall that
\begin{equation*}
x_1=\prod_{i=1}^{2k}N_i,\quad
x_1^\sigma=\prod_{i=1}^{2k}N_i^{\sigma_i},\quad
N_i\ll x_1^{19/20}\quad\forall i,
\end{equation*}
\begin{equation}
N_i\ll x_1^{1/k}\quad\text{if }i\le k,\quad
 \sigma_i\le 1 \quad\forall i.\label{Constraints}
\end{equation}
In this section we aim to use published large value estimates to obtain bounds on $R(T)$ and $R^*(T)$. Specifically we aim to prove the following proposition.	
\begin{prpstn}\label{prop}
One of the following holds uniformly for $T\in[T_1,T_0]$ and for any $(N_i),(\sigma_i)$ satisfying \eqref{Constraints}
\begin{align*}
(i):&\quad R(T)\ll x_1^{1-\sigma}(\log{x_1})^{-4k-2},\\
(ii):&\quad R(T)\ll T_0x_1^{5/4-2\sigma+8\epsilon},\\
(iii):&\quad R^*(T)\ll T_0x_1^{13/4-4\sigma+8\epsilon}.
\end{align*}
Hence Proposition \ref{mainprop} holds by Proposition \ref{PropMainResult} (with $\nu=1/4$) and \eqref{TauSize}.
\end{prpstn}
Heath-Brown used essentially the same argument thus far in \cite{HBII} and \cite{HBIII}, but worked with zeroes of $\zeta(s)$ instead of Dirichlet polynomials. The estimates on the density of zeroes used the zero detection method, which amounted to bounding the frequency with which Dirichlet polynomials take large values. The advantage we get from using Dirichlet polynomials throughout is that we have the additional condition that the total combined length $x_1$ of $S_1\dots S_{2k}$ is approximately $x$. If we did not have this restriction to our Dirichlet polynomials then we would only be able to produce the same result as Heath-Brown \cite{HBIII} . The critical case would have been when $\sigma_i=3/4$ $\forall i$, $x_1=\tau^{9/5}$ and $N_i=\tau^{2/5}$ or $1/2$ $\forall i$. But we cannot have a set of Dirichlet polynomials each with length $\tau^{2/5}$ and combined length $\tau^{9/5}$, and so the critical case cannot exist when we have this additional constraint. This allows us to improve the overall result.

We put
\begin{equation}
\mu=\frac{\log{x_1}}{\log{T_0}}
\end{equation}
to simplify notation. we note that since $x\ll x_1\ll x$, inequality \eqref{TauSize} implies that we only need consider
\begin{equation}
\frac{4}{3}\le \mu\le \frac{19}{9}.
\end{equation}
\subsection{Initial Estimates}
Our proof will make extensive use of the following three bounds on the frequency of large values taken by Dirichlet polynomials.

We consider a Dirichlet polynomial $S(t)=\sum_N^{2N}a_nn^{-c+it}$ which is a product of some of the $S_i$. Therefore $S=\prod_{i\in\mathcal{I}}S_i$, $N=\prod_{i\in\mathcal{I}}N_i$ and $N^{\sigma'}=\prod_{i\in\mathcal{I}}N_i^{\sigma_i}$ for some set $\mathcal{I}\subset \{1,\dots,2l\}$. We let
\begin{align*}
R(S;T)&=\#\biggl\{m\in\mathbb{Z}\cap[T,2T]:N^{-c+\sigma'}\le\sup_{t\in[m,m+1]}|S(t)|\le 2N^{-c+\sigma'}\biggr\},\\
R^*(S;T)&=\#\biggl\{(m_1,m_2,m_3,m_4)\in(\mathbb{Z}\cap[T,2T])^4:m_1+m_2=m_3+m_4,\\
&\qquad\qquad\qquad\qquad\qquad\qquad N^{-c+\sigma'}\le \sup_{t\in[m_i,m_i+1]}|S(t)|\le 2N^{-c+\sigma'}\forall i\biggr\}.
\end{align*}
Clearly we have $R(T)\le R(S;T)$ and $R^*(T)\le R^*(S;T)$. We note that the coefficients $a_n$ of $S$ satisfy $a_n=O_\delta(T_0^\delta)$ for every $\delta>0$.
\begin{lmm}[Montgomery's Mean Value Estimate]\label{MontMV}
We have
\[R(S;T)\ll (\log{N T})\left(N^{2-2\sigma'}+TN^{1-2\sigma'}\right)\left(\frac{\sum_{N}^{2N}|a_n|^2}{N}\right).\]
In particular, uniformly for $T_1\le T\le T_0$ and for any $\delta>0$, we have
\[R(T)\ll_\delta T_0^{\delta}N^{2-2\sigma'}+T_0^{1+\delta}N^{1-2\sigma'}.\]
\end{lmm}
\begin{proof}
The first statement is proven in \cite{Montgomery}[Theorem 7.3]. The second statement follows immediately from the first since $N\le x_1\le T_0^3$.
\end{proof}
\begin{lmm}[Huxley's Large Values Estimate]\label{HuxLV}
We have
\[R(S;T)\ll (\log{NT})^2\left(N^{2-2\sigma}+TN^{4-6\sigma}\right)\left(1+\frac{\sum_{N}^{2N}|a_n|^2}{N}\right)^3.\]
In particular, uniformly for $T_1\le T\le T_0$ and for any $\delta>0$, we have 
\[R(T)\ll_\delta T_0^{\delta}N^{2-2\sigma'}+T_0^{1+\delta}N^{4-6\sigma'}.\]
\end{lmm}
\begin{proof}
The first statement is proven in \cite{Huxley}[Equation 2.9]. The second statement follows immediately from the first since $N\le x_1\le T_0^3$.
\end{proof}
\begin{lmm}[Heath-Brown's $R^*$ Bound]\label{HBR*}
For any $\delta>0$ we have
\begin{align*}
R^*(S;T)&\ll_\delta N^{1-2\sigma'}T^{\delta}(R(S;T)N+R(S;T)^2+R(S;T)^{5/4}T^{1/2})^{1/2}\\
&\qquad\times(R^*(S;T)N+R(S;T)^4+R(S;T)R^{*}(S;T)^{3/4}T^{1/2})^{1/2}.
\end{align*}
In particular, uniformly for $T_1\le T\le T_0$ and for any $\delta>0$, we have
\begin{align*}
R^*(T)&\ll_\delta N^{1-2\sigma'}T_0^{\delta}(R(T)N+R(T)^2+R(T)^{5/4}T_0^{1/2})^{1/2}\\
&\qquad\times(R^*(T)N+R(T)^4+R(T)R^{*}(T)^{3/4}T_0^{1/2})^{1/2}
\end{align*}
\end{lmm}
\begin{proof}
The first statement is proven in \cite{HBZeroDensity}[Equation 33]. The second statement follows from a precisely analogous argument applied to $R^*(T)$.
\end{proof}
In addition to these results, we will also require the following lemma.
\begin{lmm}\label{ZetaBounds}
Either
\[R(T)\ll x^{1-\sigma}(\log{x})^{-4k-2}\]
or for $k<i\le 2k$ we have
\[R(T)\ll (\log{T})^{41}T^2N_i^{6-12\sigma_i}\quad\text{and}\quad R(T)\ll (\log{T})^{13}TN_i^{2-4\sigma_i}.\]
\end{lmm}
\begin{proof}
We follow the method of Heath-Brown in \cite{HBHuxley} but making use of the twelfth as well as the fourth power moment of the Zeta function.

We consider a polynomial $S_i$ with $k<i\le 2k$. Such a polynomial has all coefficients $1$ (if $i<2k$) or all coefficients $\log{n}$ (if $i=2k$). We first consider the case when the coefficients of $S_i$ are identically $1$.

From Perron's formula with $T_1\le T \le T_0$ we have for $T\le t\le 2T$ that
\begin{align*}
\left|\sum_N^{2N}n^{-1/2-it}\right|&=\left|\int_{d-iT/2}^{d+iT/2}\zeta(1/2+it+s)\frac{(2N)^s-N^s}{s}ds\right|+O(N^{1/2}T^{-1}(\log{x_1})+1)
\end{align*}
where $d=1/2+(\log{x_1})^{-1}$.

Moving the line of integration to $\Re(s)=0$ gives
\[\left|\sum_N^{2N}n^{-1/2-it}\right|\ll\int^{5T/2}_{T/2}\left|\zeta(1/2+iu)\right|\frac{du}{1+|t-u|}+N^{1/2}T^{-1}(\log{x_1})+1.\]
Let $(m_j)_1^{R(T)}\subset\mathbb{Z}\cap[T,2T]$ be such that
\[\sup_{t\in[m_j,m_j+1]}\left|\sum_N^{2N}n^{-1/2-it}\right|\gg N^{\sigma-1/2}.\]
Let $t_j$ be a point in $[m_j,m_j+1]$ where this supremum is attained.

Then, using H\"older's inequality and Heath-Brown's twelfth power moment bound for $\zeta(s)$ (see \cite{TwelfthPower}[Theorem 1]):
\begin{align*}
R(T)N^{12\sigma-6}&\ll\sum_{1\le j\le R(T)}\left|\sum_{N}^{2N}n^{-1/2+it_j}\right|^{12}\\
&\ll \sum_{1\le j\le R(T)}\left(\int_{T}^{2T}\left|\zeta(1/2+iu)\right|^{12}\frac{du}{1+|t_j-u|}\right)\left(\int_{T}^{2T}\frac{du}{1+|t_j-u|}\right)^{11}\\
&\qquad\qquad+R(T)N^6T^{-12}(\log{x_1})^{12}+R(T)\\
&\ll (\log{x_1})^{11}\int_{T}^{2T}\left|\zeta(1/2+iu)\right|^{12}\sum_{1\le j\le R(T)}\frac{1}{1+|t_j-u|}du\\
&\qquad\qquad+R(T)N^6T^{-12}(\log{x_1})^{12}+R(T)\\
&\ll T^2(\log{x_1})^{29}+R(T)N^6T^{-12}(\log{x_1})^{12}+R(T)\\
&\ll T^2(\log{x_1})^{29}+R(T)N^6T^{-12}(\log{x_1})^{12},
\end{align*}
since $R(T)\ll T$.

In the case $i=2k$ and all coefficients are $\log{n}$ we obtain by partial summation and the method above
\[R(T)N^{12\sigma-6}\ll T^2(\log{x_1})^{41}+R(T)N^{6}T^{-12}(\log{x_1})^{24}.\]
In either case we get
\begin{equation}\label{TwelfthPM}
R(T)\ll (\log{x_1})^{41}\left(T^2N^{6-12\sigma}+R(T)N^{12-12\sigma}T^{-12}\right).
\end{equation}
We can apply the same method, but using the fourth power moment of $\zeta(s)$ (see \cite{FourthPower}[Theorem B] for example) instead of the twelfth. We obtain (for both $\sum_N^{2N}n^{-1/2-it}$ and $\sum_N^{2N}(\log{n})n^{-1/2-it}$)
\begin{align*}
R(T)N^{4\sigma-2}&\ll (\log{x_1})^8\left( \int_{T/2}^{5T/2}\left|\zeta(1/2+iu)\right|^4\sum_{1\le j\le R(T)}\frac{1}{1+|t_j-u|}du\right)\\
&\qquad+(\log{x_1})^8\left(R(T)N^2T^{-6}+R(T)\right)\\
&\ll (\log{x_1})^{13}\left(T+R(T)N^2T^{-4}\right).
\end{align*}
Thus
\begin{equation}\label{FourthPM}
R(T)\ll (\log{x_1})^{13}\left(TN^{2-4\sigma}+R(T)N^{4-4\sigma}T^{-4}\right).
\end{equation}
From \eqref{TwelfthPM} and \eqref{FourthPM} we see one of the following must hold for any Dirichlet polynomial $S_j$ with $i>k$:\\
(i): $T \ll (\log{x_1})^4N_j^{1-\sigma_j}$\\
(ii): $R(T)\ll (\log{x_1})^{41}T^{2}N_j^{-6(2\sigma_j-1)}$ and $R(T)\ll (\log{x_1})^{13}TN_j^{2-4\sigma_j}$.

We are therefore left to show that $(i)$ implies that $R(T)\ll x_1^{1-\sigma}(\log{x_1})^{-4k-2}$.

We note that
\[R(T)\ll T\ll (\log{x_1})^4N_j^{1-\sigma_j}\ll(\log{x_1})^4x_1^{1-\sigma}\prod_{i\ne j}N_i^{\sigma_i-1}.\]
This is good enough to prove
\[R(T)\ll x_1^{1-\sigma}(\log{x_1})^{-4k-2}\]
provided that for some $j'\ne j$ we have
\[N_{j'}^{1-\sigma_{j'}}\gg (\log{x_1})^{4k+6}.\]
Since $N_i\ll x_1^{19/20}$ $\forall i$ there must be some $j'\ne j$ such that $N_{j'}\gg x_1^{1/40k}$ (since there are $2k$ polynomials whose combined length $\prod N_i$ is $x_1$). Thus we need to show that $\sigma_{j'}$ cannot be too close to $1$.

We put
\begin{equation}\label{EtaDef}
\eta=\eta(x_1)=C_0(\log{x_1})^{-2/3}(\log\log{x_1})^{-1/3}
\end{equation}
for some suitable constant $C_0>0$ (which we will declare later).

By Perron's formula we have for $t\in[T,2T]$ that
\begin{align*}
\left|S_{j'}(c+it)\right|&=\frac{1}{2\pi i}\int^{iT/2}_{-iT/2}F_{j'}(c+it+s)\frac{(2N_{j'})^s-N_{j'}^s}{s}ds+O(T^{-1}\log{x_1})
\end{align*}
where
\[F_{j'}(s)=
\begin{cases}
(\zeta(s))^{-1},&1\le j'\le k\\
\zeta(s),\qquad &k<j'<2k\\
\zeta'(s),&j'=2k\\
\end{cases}\]
In the region $1-2\eta-c\le \Re(s)\le 0$, $|t-\Im(s)|\le T/2$ we have
\[\left|F_{j'}(c+it+s)\right|\ll (\log{x_1})^2\]
for any $1\le j'\le 2k$. This follows from \cite{Titchmarsh}[Theorem 3.11] along with the Vinogradov-Korobov estimate as given in \cite{Richert} for a suitable choice of $C_0$.

We now move the line of integration to $\Re(s)=1-2\eta-c$ to obtain
\begin{align*}
|S_{j'}(c+it)|&\ll \int^{5T/2}_{T/2}|F_j(1-2\eta+it+is)|N_{j'}^{-2\eta}|s|^{-1}ds+O(T^{-1}\log{x_1})\\
&\ll (\log{x_1})^3(N_{j'}^{-2\eta}+T_1^{-1}).
\end{align*}
Thus, since $N_{j'}>x_1^{1/40k}$, we have $N_{j'}^{\eta/2}\gg (\log{x_1})^{4}$. This gives
\[|S_{j'}(c+it)|\le N_i^{-3\eta/2}.\]
Therefore we have
\begin{equation}\label{VingradovSigmaNearOne}
R=0\qquad\text{or}\qquad\sigma_{j'}\le 1-3\eta/2
\end{equation}
for any polynomial with $N_{j'}>x_1^{1/40k}$.

In particular either
\[R=0\ll x_1^{1-\sigma}(\log{x})^{-4k-2}\]
or
\[N_{j'}^{1-\sigma_{j'}}\gg \left(x_1^{1/40k}\right)^{3\eta/2}\gg (\log{x_1})^{4k+6}\]
which implies that
\[R(T)\ll x_1^{1-\sigma}(\log{x_1})^{-4k-6}.\]
Thus the lemma holds.
\end{proof}
We now summarise the other large-value estimates which we will make use of, which follow from published work by other authors.
\begin{lmm}\label{Published}
Either:
\begin{align*}
R(T)\ll x_1^{1-\sigma}(\log{x_1})^{-4k-2}
\end{align*}
or:\\
uniformly for $T_1\le T\le T_0$ we have 
\begin{align}
R(T)&\ll_\delta T_0^{(3-3\sigma)/(2-\sigma)+\delta},&\quad &\text{if }\sigma\le 3/4,\label{MontResult}\\
R(T)&\ll_\delta T_0^{(3-3\sigma)/(3\sigma-1)+\delta},&&\text{if }\sigma\ge 3/4,\label{HuxResult}\\
R(T)&\ll_\delta T_0^{(3-3\sigma)/(10\sigma-7)+\delta},&&\text{if }\sigma\le 25/28,\label{HB1Result}\\
R(T)&\ll_\delta T_0^{(4-4\sigma)/(4\sigma-1)+\delta},&&\text{if }\sigma\ge 25/28,\label{HB2Result}\\
R^*(T)&\ll_\delta  T_0^{(15-16\sigma)/2+\delta},&&\text{if }\sigma\le 3/4,\label{HB3Result}\\
R^*(T)&\ll_\delta  T_0^{(12-12\sigma)/(4\sigma-1)+\delta},&&\text{if }\sigma\ge 3/4\label{HB4Result}
\end{align}
for any $\delta>0$.
\end{lmm}
\begin{proof}
We assume that $R\ll x_1^{1-\sigma}(\log{x_1})^{-4k-2}$ does not hold.

These bounds are usually obtained merely as an intermediate step in the zero detection method when trying to bound $N(\sigma,T)$ (or $N^*(\sigma,T)$). They are therefore not always explicitly stated as a lemma in the papers where they are obtained.

Since the published bounds all bound Dirichlet polynomials which arise from the zero detection method, they do not immediately apply to our context, since the Dirichlet polynomials we are considering can be slightly different. In particular, the results we will quote only apply to a Dirichlet polynomial $S$ with length $N\in [Y^{1/2},Y]$ (for some value of $Y\le T_0$) and coefficients which are $O_\delta(T_0^\delta)$ for every $\delta>0$.

We repeatedly combine any pair of polynomials of length $\le T_0^\delta$, so that there is at most one polynomial of length $\le T_0^\delta$. We only need to consider $\delta<1/(2k)$, and so any polynomial with length $\ge T_0^{3/k}$ must have all coefficients $1$ all coefficients $\log{n}$. This means that the $R$-bounds of Lemma \ref{ZetaBounds} still apply to any of the polynomials with length $\ge T_0^{3/k}$ after these combinations.

We pick a polynomial $S_{j_1}$ of length $N_{j_1}>T_0^\delta$ with $\sigma_{j_1}$ maximal.

If $\sigma_{j_1}< \sigma$ then, since $\sigma$ is an average of the $\sigma_i$, the polynomial $S_{j_2}$ with length $N_{j_2}\le T_0^\delta$ must exist and have $\sigma_{j_2}> \sigma$. In this case we combine the polynomials $S_{j_1}$ and $S_{j_2}$ to produce a polynomial $S$ of length $N$ and size $\sigma'\ge\sigma$.

If $\sigma_{j_1}\ge\sigma$ we take $S=S_{j_1}$ (and so $N=N_{j_1}$ and $\sigma'=\sigma$).

Since the bounds \eqref{MontResult}, \eqref{HuxResult}, \eqref{HB1Result}, \eqref{HB2Result}, \eqref{HB3Result} and , \eqref{HB4Result} are all decreasing in $\sigma$, it is sufficient to prove them for $R(S;T)$ and $R^*(S;T)$ when $\sigma'=\sigma$.

If $N\le Y$, then by raising the polynomial $S$ to a suitable exponent we can ensure that the new polynomial, $S'$ say, has length $N'\in [Y^{1/2},Y]$. The Dirichlet polynomials $S_i$ which we are considering (or any combination of them) have coefficients which are $O_\delta(T_0^\delta)$ for every $\delta>0$. Therefore the coefficients of $S'$ will also be $O_\delta(T_0^\delta)$ for every $\delta>0$ provided we have raised $S$ to an exponent which is $O_\delta(1)$. This is the case since by construction we have $N>T_0^\delta$. Therefore the published bound will hold if $N\le Y$.

If $N\ge Y$ then we will use Lemma \ref{ZetaBounds} to obtain the result (recalling that we have assumed that $R\ll x_1^{1-\sigma}(\log{x_1})^{-4k-2}$ does not hold). If $N=N_{j_1}$ or $N=N_{j_1}N_{J_2}$ then by choosing $k$ large enough we must have $N_{j_1}>T_0^{3/k}$, and so Lemma \ref{ZetaBounds} applies. If $N=N_{j_1}N_{j_2}$ then we have
\begin{equation}\label{ZBound1}
R(T)\ll R(T)T_0^{2\delta}N_{j_2}^{2-4\sigma_{j_2}}\ll_\delta T_0^{1+3\delta}N_{j_1}^{2-4\sigma_{j_1}}N_{j_2}^{2-4\sigma_{j_2}}\ll_\delta T_0^{1+3\delta}N^{2-4\sigma}
\end{equation}
and
\begin{equation}\label{ZBound2}
R(T)\ll R(T)T_0^{6\delta}N_{j_2}^{6-12\sigma_{j_2}} \ll_\delta T_0^{1+7\delta}N^{6-12\sigma}
\end{equation}
for any $\delta>0$. We see that \eqref{ZBound1} and \eqref{ZBound2} trivially follow from Lemma \ref{ZetaBounds} if $N=N_{j_1}$, and so they hold in either case.

We now establish \eqref{MontResult}, \eqref{HuxResult}, \eqref{HB1Result}, \eqref{HB2Result}, \eqref{HB3Result} and \eqref{HB4Result} in turn.

We see that \eqref{MontResult} holds trivially if $\sigma\le 2/3$. In the proof of Theorem 12.1 in \cite{Montgomery}, Montgomery shows that $R(S;T)\ll_\delta T_0^{(3-3\sigma)/(2-\sigma)+\delta}$ if $S$ has length $N\in[Y^{1/2},Y]$ with $Y=T_0^{3/(8-4\sigma)}$ and $1/2\le\sigma\le 3/4$. Therefore \eqref{MontResult} holds if $N\le T_0^{3/(8-4\sigma)}$ (since $N>T_0^\delta$). If $N\ge T_0^{3/(8-4\sigma)}$ then since $k\ge 9$ we must have $N_{j_1}>T_0^{3/k}$. Then by \eqref{ZBound1} we have
\[R(T)\ll_\delta T_0^{1+3\delta}N^{2-4\sigma}\ll T_0^{(7-8\sigma)/(4-2\sigma)+3\delta}\ll T_0^{(3-3\sigma)/(2-\sigma)+3\delta}\]
for any $\delta>0$ (since we only need to consider $\sigma\ge 2/3$). This establishes \eqref{MontResult}.

In the proof of inequality (19) in \cite{Huxley}, Huxley shows that $R(S;T)\ll_\delta T_0^{(3-3\sigma)/(3\sigma-1)+\delta}$ if $S$ has length $N\in[Y^{1/2},Y]$ with $Y=T_0^{3/(12\sigma-4)}$. Therefore \eqref{HuxResult} holds if $N\le T_0^{3/(12\sigma-4)}$ (since $N>T_0^\delta$). If $N\ge T_0^{3/(12\sigma-4)}$ then since $k\ge 9$ we must have $N_{j_1}>T_0^{3/k}$. Then by \eqref{ZBound2} we have
\[R(T)\ll_\delta T_0^{2+7\delta}N^{6-12\sigma}\ll T_0^{(5-6\sigma)/(6\sigma-2)+7\delta}\ll T_0^{(3-3\sigma)/(3\sigma-1)+7\delta}\]
for any $\delta>0$. This establishes \eqref{HuxResult}.

In the proof of Theorem 1 in \cite{HBZeroDensity}, Heath-Brown proves $R(S;T)\ll_\delta T_0^{(3-3\sigma)/(10\sigma-7)+\delta}$ if $S$ has length $N\in[Y^{1/2},Y]$ with $Y=T_0^{3/(40\sigma-28)}$ and $\sigma\le 25/28$. Therefore \eqref{HB1Result} holds if $N\le T_0^{3/(40\sigma-28)}$ (since $N>T_0^\delta$). If $N\ge T_0^{3/(40\sigma-28)}$ then since $k\ge 13$ we must have $N_{j_1}>T_0^{3/k}$. Then by \eqref{ZBound2} we have
\[R(T)\ll_\delta T_0^{2+7\delta}N^{6-12\sigma}\ll T_0^{(22\sigma-19)/(20\sigma-14)+7\delta}\ll T_0^{(3-3\sigma)/(10\sigma-7)+7\delta}\]
for any $\delta>0$ (since we are only considering $\sigma\le 25/28$ in \eqref{HB1Result}). This establishes \eqref{HB1Result}.

In the proof of Theorem 1 in \cite{HBZeroDensity}, Heath-Brown shows that $R(S;T)\ll_\delta T_0^{(4-4\sigma)/(4\sigma-1)+\delta}$ if $S$ has length $N\in[Y^{1/2},Y]$ with $Y=T_0^{1/(4\sigma-1)}$ and $\sigma\ge 25/28$. Therefore \eqref{HB2Result} holds if $N\le T_0^{1/(4\sigma-1)}$ (since $N>T_0^\delta$). If $N\ge T_0^{1/(4\sigma-1)}$ then since $k\ge 10$ we must have $N_{j_1}>T_0^{3/k}$. Then by \eqref{ZBound2} we have
\[R(T)\ll_\delta T_0^{2+7\delta}N^{6-12\sigma}\ll T_0^{(4-4\sigma)/(4\sigma-1)+7\delta}\]
for any $\delta>0$. This establishes \eqref{HB2Result}.

We see that \eqref{HB3Result} holds trivially if $\sigma\le 1/2$. In the proof of Theorem 2 in \cite{HBZeroDensity}, Heath-Brown shows that $R^*(S;T)\ll_\delta T_0^{(10-11\sigma)/(2-\sigma)+\delta}+T_0^{(18-19\sigma)/(4-2\sigma)+\delta}$ if $S$ has length $N\in[Y^{1/2},Y]$ with $Y=T_0^{1/2}$ and $1/2\le \sigma\le 3/4$. In particular, this gives $R(S;T)\ll_\delta T_0^{(15-16\sigma)/2+\delta}$ for any $\delta>0$ and $\sigma\le 3/4$. Therefore \eqref{HB3Result} holds if $N\le T_0^{1/2}$  (since $N>T_0^\delta$). If $N\ge T_0^{1/2}$ then since $k\ge 7$ we must have $N_{j_1}>T_0^{3/k}$. In this case, using the trivial bound $R^*(T)\ll (\log{T_0})R(T)^3$ and \eqref{ZBound1}, we have
\[R^*(T)\ll_\delta (\log{T_0})(T_0^{1+3\delta/4}N^{2-4\sigma})^3\ll T_0^{6-6\sigma+10\delta}\ll T_0^{(15-16\sigma)/2+10\delta}\]
for any $\delta>0$ (since we are only considering $\sigma\le 3/4$ in \eqref{HB3Result}). This establishes \eqref{HB3Result}.

In the proof of Theorem 2 in \cite{HBZeroDensity}, Heath-Brown proves $R^*(S;T)\ll_\delta T_0^{(12-12\sigma)/(4\sigma-1)+\delta}$ if $S$ has length $N\in[Y^{1/2},Y]$ with $Y=T_0^{1/(4\sigma-1)}$ and $\sigma\ge 3/4$. Therefore \eqref{HB4Result} holds if $N\le T_0^{1/(4\sigma-1)}$  (since $N>T_0^\delta$). If $N\ge T_0^{1/(4\sigma-1)}$ then since $k\ge 10$ we must have $N_{j_1}>T_0^{3/k}$. Then by \eqref{ZBound2} we have
\[R^*(T)\ll (\log{T_0})R(T)^3\ll_\delta (\log{T_0})(T_0^{2+7\delta}N^{6-12\sigma})^3\ll T_0^{(12-12\sigma)/(4\sigma-1)+22\delta}\]
for any $\delta>0$. This establishes \eqref{HB4Result}.
\end{proof}
To simplify notation we drop the $T$ from $R$ and $R^*$ since we are from now on only interested in them evaluated at $T$. Thus
\[R=R(T),\qquad R^*=R^*(T).\]
We now prove Proposition \ref{prop} by way of five lemmas. Lemma \ref{LargePolys} covers the case when some of the polynomials are long by using Lemma \ref{ZetaBounds}. Lemma \ref{ssmall} covers the case $\sigma\le 3/4$ by using Montgomery's mean-value estimate and Heath-Browns $R^*$ estimate. Lemma \ref{slarge} covers the case $\sigma\ge 3/4$ and $\mu$ `small' using the same method but using Huxley's large values estimate and Heath-Brown's $R^*$ estimate. Lemma \ref{mu} covers the case when $\sigma>3/4$ and $\mu$ is `large' using an adapted argument from \cite{HBZeroDensity} and Lemma \ref{ZetaBounds}. Lemma \ref{SClose1} deals with the range when $\sigma$ is very close to $1$ using Vinogradov's zero-free region of $\zeta(s)$ and Van-der-Corput's method of exponential sums.
%\newpage
\subsection{Part 1: Long Polynomials}
We first notice that we only need to consider polynomials of reasonably short length, where published estimates for the frequency with which they take large values apply.
\begin{lmm}\label{LargePolys}
Either we have one of
\[R\ll T_0x_1^{5/4-2\sigma+\epsilon},\quad R\ll x_1^{1-\sigma}(\log{x_1})^{-4k-2}\]
or we have
\begin{equation}
N_i\le T_0^{1/2+\epsilon}\label{ZetaPieceBound}
\end{equation}
for all but at most one $i$. If such an exceptional polynomial $S_j$ exists then $T_0^{1/2+\epsilon}\le N_j\le x_1^{3/5}$.
\end{lmm}
\begin{proof}
We assume that $R\ll x_1^{1-\sigma}(\log{x})^{-4k-2}$ does not hold. Therefore the results of Lemmas \ref{ZetaBounds} and \ref{Published} apply.

We consider polynomials $S_i$ with $N_i>T_0^{1/2}$. Since we are taking $k\ge 6$, by \eqref{MuPieceBound} any such polynomial with `long' length must be one where all coefficients are 1 or $\log(n)$. This means we can use Lemma \ref{ZetaBounds} to get stronger than normal bounds.

\textbf{Case 1: }There are at least 2 such values of $j$ such that $N_j>T_0^{1/2+\epsilon}$.

Let $j_1,j_2$ be two values of $j$ such that $N_j>T_0^{1/2+\epsilon}$. We let $N=N_{j_1}N_{j_2}$ ($>T_0^{1+2\epsilon}$), $N^\alpha=N_{j_1}^{\sigma_{j_1}}N_{j_2}^{\sigma_{j_2}}$, $M=\prod_{i\ne j_1,j_2}N_i$, $M^\beta=\prod_{i\ne j_1,j_2}N_i^{\sigma_i}$. By Lemmas \ref{MontMV}, \ref{HuxLV} and \ref{ZetaBounds} we have for any $\delta>0$ that
\begin{align*}
R&\ll_\delta (T_0^{1+\delta}N_{j_1}^{2-4\sigma_{j_1}})^{1/2}(T_0^{1+\delta}N_{j_2}^{2-4\sigma_{j_2}})^{1/2}=T_0^{1+\delta}N^{1-2\alpha},\\
R&\ll_\delta (T_0^{2+\delta}N_{j_1}^{6-12\sigma_{j_1}})^{1/2}(T_0^{2+\delta}N_{j_2}^{6-12\sigma_{j_2}})^{1/2}=T_0^{2+\delta}N^{3-6\alpha},\\
R&\ll_\delta M^{2-2\beta}T_0^{\delta}+\min(T_0^{1+\delta}M^{1-2\beta},T_0^{1+\delta}M^{4-6\beta}).
\end{align*}
We choose $\delta=\epsilon/2$, and so the implied constants only need to depend on $\epsilon$.

If $R\ll M^{2-2\beta}T_0^{\epsilon/2}$ then since $N\ge T^{1+2\epsilon}$ we have
\begin{align*}
R&\ll (T_0^{1+\epsilon/2}N^{1-2\alpha})^{1/2}(T_0^{\epsilon/2}M^{2-2\beta})^{1/2}\\
&=T_0^{1/2+\epsilon/2}N^{-1/2}x_1^{1-\sigma}\\
&\ll x_1^{1-\sigma}T_0^{-\epsilon/2}\\
&\ll x_1^{1-\sigma}(\log{x_1})^{-4k-2}.
\end{align*}
If $R\ll \min(T_0^{1+\epsilon/2}M^{1-2\beta},T_0^{1+\epsilon/2}M^{4-6\beta})$ then since $N\ge T$
\begin{align*}
R&\ll (T_0^{1+\epsilon/2}M^{1-2\beta})^{1/4}(T_0^{1+\epsilon/2}M^{4-6\beta})^{1/4}(T_0^{1+\epsilon/2}N^{1-2\alpha})^{1/4}(T_0^{2+\epsilon/2}N^{3-6\alpha})^{1/4}\\
&=T_0^{5/4+\epsilon/2}N^{-1/4}x_1^{5/4-2\sigma}\\
&\ll T_0x_1^{5/4-2\sigma}.
\end{align*}
Therefore if there are two polynomials with length $\ge T_0^{1/2+\epsilon}$ then
\[R\ll x_1^{1-\sigma}(\log{x_1})^{-4k-2}\qquad\text{or}\qquad R\ll T_0x_1^{5/4-2\sigma+\epsilon}.\]
\textbf{Case 2: }There is a $j$ such that $N_j>x_1^{3/5}$.

We consider the long polynomial $S_j$ and its complement. To ease notation we let $N=N_j$, $N^\alpha=N_j^{\sigma_j}$, $M=\prod_{i\ne j}N_i$ with $M^{\beta}=\prod_{i\ne j}N_i^{\sigma_i}$. Then by Lemmas \ref{ZetaBounds} and \ref{MontMV} (choosing $\delta=\epsilon$) we have
\begin{align*}
R&\ll (M^{2-2\beta}T_0^{\epsilon}+T_0^{1+\epsilon}M^{1-2\beta})^{6/7}(T_0^{2+\epsilon}N^{6-12\alpha})^{1/7}\\
&\ll T_0^{2/7+\epsilon}x_1^{12(1-\sigma)/7}N^{-6/7}+T_0^{8/7+\epsilon}x_1^{6/7-12\sigma/7}
\end{align*}
Since $N>x_1^{3/5}$ and $4/3\le \mu\le 19/9$ this gives
\[R\ll T_0x_1^{5/4-2\sigma+\epsilon}.\]
Therefore the Lemma holds.
\end{proof}
We note that inequalities \eqref{MuPieceBound} and \eqref{ZetaPieceBound} are vital in our treatment of the problem in this way. The $S_i$ for $i\le k$ are `difficult' since the coefficients $\mu(n)$ have complicated behaviour, but by increasing $k$ we can ensure these polynomials do not cause too many problems. This is because we have effective bounds on the number of large values reasonably short Dirichlet polynomials can take. We do not have the same method of controlling the length of $S_i$ for $i>k$, but these polynomials have `well-behaved' coefficients. This allows us to produce much stronger bounds in Lemma \ref{ZetaBounds} and so cope with the longer polynomials.

From now on we assume that $N_i\le T_0^{1/2+\epsilon}$ $\forall i$ except for possibly one exceptional polynomial $S_j$ with $T_0^{1/2+\epsilon}\le N_j\le x_1^{3/5}$.
%\newpage
\subsection{Part 2: $\sigma\le3/4$}
\begin{lmm}\label{ssmall}
Let $\sigma\le 3/4$. Then either
\[R\ll x_1^{1-\sigma}(\log{x_1})^{-4k-2}\]
or
\[R\ll T_0x_1^{5/4-2\sigma+2\epsilon}\]
or
\[R^*\ll T_0x_1^{13/4-4\sigma+6\epsilon}.\]
\end{lmm}
\begin{proof}
We assume that $R\ll x_1^{1-\sigma}(\log{x_1})^{-4k-2}$ does not hold. Therefore the results of Lemma \ref{Published} apply.

The result follows from published estimates of Lemma \ref{Published} unless $8/5\le \mu\le 2$ and $7/10\le \sigma\le 3/4$.

For $\sigma\le 5/8$ we use the trivial estimate
\[R\ll T_0\ll T_0x^{5/4-2\sigma}.\]
By \eqref{MontResult} (choosing $\delta=\epsilon$) we have
\[R\ll T_0^{(3-3\sigma)/(2-\sigma)+\epsilon}.\]
This gives $R\ll T_0x^{5/4-2\sigma+\epsilon}$ if $\mu\le 2$ and $5/8\le\sigma\le 7/10$ or if $\mu\le 8/5$ and $7/10\le\sigma\le 3/4$.

By \eqref{HB3Result} (choosing $\delta=\epsilon$) we have
\[R^*\ll T_0^{(15-16\sigma)/2+\epsilon}.\]
This gives $R\ll T_0x^{13/4-4\sigma+\epsilon}$ if $\mu\ge 2$.

These cover all ranges of $\mu$ and $\sigma$ unless $8/5\le \mu\le 2$ and $7/10\le \sigma\le 3/4$. We now consider this case.

We combine the polynomials to produce two polynomials of length $M,N$ and size $\alpha,\beta$. (so $M=\prod_{i\in I_1}N_i$, $M^\alpha=\prod_{i\in I_1}N_i^{\sigma_i}$, $N=\prod_{i\in I_2}N_i$, $N^\beta=\prod_{i\in I_2}N_i^{\sigma_i}$ for some disjoint $I_1,I_2\subset \{1,\dots 2k\}$ with $I_1\cup I_2=\{1,\dots, 2k\}$). We will declare how we combine the polynomials later. We let $M$ be the smaller of the two (so $M\le N$). Therefore we have
\[x_1=T_0^{\mu}=MN,\qquad M^{\alpha}N^{\beta}=x_1^\sigma.\]
Since $\mu\le 2$, we have $T_0^2\ge x_1=MN\ge M^2$, and so $M\le T_0$.\

By Lemma \ref{MontMV} (choosing $\delta=\epsilon$) we have
\[R\ll\min(T_0^{\epsilon}M^{2-2\alpha}+T_0^{1+\epsilon}M^{1-2\alpha},T_0^{\epsilon}N^{2-2\beta}+T_0^{1+\epsilon}N^{1-2\beta}).\]
We note that the first term in each component dominates iff the polynomial has length $\ge T_0$. Since $M\le T_0$, we always have the second term ($T_0^{1+\epsilon}M^{1-2\alpha}$) dominating the first component of the minimum. We split the argument into two cases, dependent on which term is larger in the second component of the minimum.

\textbf{Case 1: $N\le T_0^{1+2\epsilon}$}.

In this case we have
\[N^{2-2\beta}T_0^\epsilon+T_0^{1+\epsilon}N^{1-2\beta}\ll T_0^{1+3\epsilon}N^{1-2\beta}.\]
Hence
\begin{align*}
R&\ll \min(T_0^{1+\epsilon}M^{1-2\alpha},T_0^{1+3\epsilon}N^{1-2\beta})\\
&\ll (T_0^{1+\epsilon}M^{1-2\alpha})^{1/2}(T_0^{1+3\epsilon}N^{1-2\beta})^{1/2}\ll T_0x_1^{1/2-\sigma+2\epsilon}.
\end{align*}
But for $\sigma\le 3/4$ we have $\frac{1}{2}-\sigma\le \frac{5}{4}-2\sigma$. Thus
\[R\ll T_0x_1^{5/4-2\sigma+2\epsilon}.\]
\textbf{Case 2: $N> T_0^{1+2\epsilon}$}.

We have by Lemmas \ref{MontMV} and \ref{HBR*}
\[R\ll \min(N^{2-2\beta+\epsilon},T_0M^{1-2\alpha+\epsilon}),\]
\[R^*\ll M^{1-2\alpha}T_0^{\epsilon}(RM+R^2+R^{5/4}T_0^{1/2})^{1/2}(R^*M+R^4+RR^{*3/4}T_0^{1/2})^{1/2}.\]
But $RM\ge R^2,R^{5/4}T_0^{1/2}$ if $R\le M,M^{4}T_0^{-2}$.

But we have
\[R^2\le N^{2-2\beta}T_0^{1+2\epsilon}M^{1-2\alpha}=T_0^{1+2\epsilon}x_1^{2-2\sigma}M^{-1}\]
Thus $R\le MT_0^\epsilon$ if $M\ge x_1^{2(1-\sigma)/3}T_0^{1/3}$ and $R\le M^4T_0^{-2+\epsilon}$ if $M\ge x_1^{2(1-\sigma)/9}T_0^{5/9}$. Hence, if
\[M\ge \max(x_1^{2(1-\sigma)/9}T_0^{5/9},x_1^{2(1-\sigma)/3}T_0^{1/3})\] then
\[RM+R^2+R^{5/4}T_0^{1/2}\ll RMT_0^\epsilon\]
so
\[R^*\ll T_0^{3\epsilon}(M^{4-4\alpha}R+M^{3/2-2\alpha}R^{5/2}+M^{12/5-16\alpha/5}T_0^{2/5}R^{8/5}).\]
We now consider separately each of the three terms dominating.

\textbf{Case 2A: $R^*\ll T_0^{3\epsilon}M^{4-4\alpha}R$}.

\[RR^*\ll T_0^{3\epsilon}M^{4-4\alpha}R^2\ll T_0^{5\epsilon}M^{4-4\alpha}N^{4-4\beta}\ll x_1^{4-4\sigma+5\epsilon}.\]
Since $\sigma\le 3/4$, $\mu\le 2$ we have
\[(2\sigma-1/2)\mu\le 2.\]
Thus
\[RR^*\ll T_0^2x_1^{9/2-6\sigma+5\epsilon}.\]
It follows that either
\[R\ll T_0x_1^{5/4-2\sigma+2\epsilon}\]
or
\[R^*\ll T_0x_1^{13/4-4\sigma+3\epsilon}.\]
\textbf{Case 2B: $R^*\ll T_0^{3\epsilon}M^{3/2-2\alpha}R^{5/2}$}.
\begin{align*}
R^*&\ll T_0^{3\epsilon}M^{3/2-2\alpha}R^{5/2}\\
&\ll T_0^{4\epsilon}M^{3/2-2\alpha}N^{2-2\beta}\left(N^{2-2\beta}T_0^{1+2\epsilon}M^{1-2\alpha}\right)^{3/4}\\
&\ll x_1^{7(1-\sigma)/2+6\epsilon}M^{-5/4}T_0^{3/4}.
\end{align*}
But then for $M\ge x_1^{(2\sigma+1)/5}T_0^{-1/5}$ we have
\[R^*\ll T_0x_1^{13/4-4\sigma+6\epsilon}.\]
\textbf{Case 2C: $R^*\ll M^{12/5-16\alpha/5}R^{8/5}T_0^{2/5+3\epsilon}$}.
\[R^*\ll M^{12/5-16/5\alpha}N^{16/5-16/5\alpha}T_0^{2/5+5\epsilon}\ll x_1^{16(1-\sigma)/5+5\epsilon}T_0^{2/5}M^{-4/5}\]
But then for $M>x_1^{\sigma-1/16}T_0^{-3/4}$ we have
\[R^*\ll T_0x_1^{13/4-4\sigma+5\epsilon}.\]
Therefore the Lemma holds, provided that we can always combine polynomials to ensure that
\[M>x_1^{\sigma-1/16}T_0^{-3/4},x_1^{(2\sigma+1)/5}T_0^{-1/5},x_1^{2(1-\sigma)/9}T_0^{5/9},x_1^{2(1-\sigma)/3}T_0^{1/3}.\]
We claim that we can always combine polynomials to ensure that the smaller polynomial $M$ satisfies $M\ge \min(x_1^{2/5},x_1/T_0^{1+2\epsilon})$. It suffices to find a product $P$ of polynomials with length in the interval $[\min(x_1/T_0^{1+2\epsilon},x_1^{2/5}),\max(T_0^{1+2\epsilon},x_1^{3/5})]$ since then either $P$ or the complementary product will have suitable length. To obtain $P$ we combine polynomials $S_i$ which are not the exceptional polynomial in decreasing order of length until we find the first product, $S^{(1)}S^{(2)}\dots S^{(r)}$ say, with length $\ge \min(x_1/T_0^{1+2\epsilon},x_1^{2/5})$. Since the exceptional polynomial has length $\le x_1^{3/5}$ such a product exists. We let $S^{(i)}$ have length $L_i$. Therefore $L_1\dots L_{r}>\min(x_1/T_0^{1+2\epsilon},x_1^{2/5})$ and so we have found a suitable product unless $L_1\dots L_{r}>\max(T_0^{1+2\epsilon},x_1^{3/5})$. Since $L_i\le T_0^{1/2+\epsilon}$ for all $i$ this means we must have $r\ge 3$. By construction we must also have that $L_1\dots L_{r-1}<\min(x_1/T_0^{1+2\epsilon},x_1^{2/5})$ and so $L_r\ge (L_1\dots L_r)/(L_1\dots L_{r-1})\ge x_1^{1/5}$. Since by construction $L_i\ge L_r\ge x_1^{1/5}$ for all $i<r$ we have that $L_1\dots L_{r-1}\ge x_1^{(r-1)/5}\ge x_1^{2/5}$. But this is a contradiction with $L_1\dots L_{r-1}<\min(x_1/T_0^{1+2\epsilon},x_1^{2/5})$, and so we must have that $L_1\dots L_r\in [\min(x_1/T_0^{1+2\epsilon},x_1^{2/5}),\max(T_0^{1+2\epsilon},x_1^{3/5})]$.

Since in the case we are considering $N> T_0^{1+2\epsilon}$, we have $M< x_1/T_0^{1+2\epsilon}$. We also have that $M\ge \min(x_1/T_0^{1+2\epsilon},x_1^{2/5})$ by the above construction. Therefore we must have $\mu>5/3$. For $\mu>5/3$ and $\sigma\ge 0.7$, we have
\[M\ge x_1^{2/5}\ge x_1^{11/16}T_0^{-3/4},x_1^{1/2}T_0^{-1/5},x_1^{2(1-\sigma)/9}T_0^{5/9},x_1^{2(1-\sigma)/3}T_0^{1/3}\]
and so the Lemma holds.
\end{proof}

%\newpage
\subsection{Part 3: $3/4\le\sigma\le 1$, $\mu$ small}
We now consider the range $3/4\le \sigma\le 1$, $\mu\le 4/(4\sigma-1)+\epsilon$.

\begin{lmm}\label{slarge}Let $3/4\le\sigma$ and $4/3\le\mu\le\frac{4}{4\sigma-1}+\epsilon$. Then we have
\[R\ll x_1^{1-\sigma}(\log{x_1})^{-4k-2}\]
or
\[R\ll T_0x_1^{5/4-2\sigma+2\epsilon}\]
or
\[R^*\ll T_0x_1^{13/4-4\sigma+8\epsilon}.\]
\end{lmm}
\begin{proof}
We assume that $R\ll x_1^{1-\sigma}(\log{x_1})^{-4k-2}$ does not hold. Therefore the results of Lemma \ref{Published} apply.

The result follows from published estimates if $\sigma\ge 13/16$ or if $\mu\le 8/5$.

By \eqref{HB4Result} (choosing $\delta=\epsilon$) we have
\[R^*\ll T_0^{12(1-\sigma)/(4\sigma-1)+\epsilon}\ll T_0T_0^{(13-16\sigma)/(4\sigma-1)+\epsilon}.\]
This gives $R^*\ll T_0 x_1^{13/4-4\sigma+\epsilon}$ for $\sigma\ge 13/16$ since $\mu\le 4/(4\sigma-1)+\epsilon$. Thus without loss of generality we assume $\sigma\le 13/16$.

By \eqref{HuxResult} (choosing $\delta=\epsilon$) we have
\[R\ll T_0^{3(1-\sigma)/(3\sigma-1)+\epsilon}.\]
This gives $R \ll T_0x_1^{5/4-2\sigma+\epsilon}$ provided that
\[\mu\le \frac{6\sigma-4}{(3\sigma-1)(2\sigma-5/4)}.\]
For $3/4\le \sigma\le 13/16$ this covers the range $\mu \le 8/5$. Therefore without loss of generality we assume $\mu \ge 8/5$.

We now consider the remaining range $3/4\le \sigma\le 13/16$ and $\mu\ge 8/5$ in the same manner as our argument in Part 2.

We combine the polynomials into two polynomials $M,N$ as in Lemma \ref{ssmall}. Therefore we can choose $M$ such that $\min(x_1T_0^{-1-2\epsilon},x_1^{2/5})\le M\le N$.

We use Lemma \ref{HuxLV} (choosing $\delta=\epsilon$) to give
\[R\ll\min(T_0^{\epsilon}M^{2-2\alpha}+T_0^{1+\epsilon}M^{4-6\alpha},T_0^{\epsilon}N^{2-2\beta}+T_0^{1+\epsilon}N^{4-6\beta}).\]
We split our argument into four cases, dependent on which terms dominate in this estimate.

\textbf{Case 1: $N^{4\beta-2}\le T_0$, $M^{4\alpha-2}\le T_0$}.
\begin{align*}
R&\ll \min(T_0^{1+\epsilon}M^{4-6\alpha},T_0^{1+\epsilon}N^{4-6\beta})\\
&\ll (T_0^{1+\epsilon}M^{4-6\alpha})^{1/2}(T_0^{1+\epsilon}N^{4-6\beta})^{1/2}\\
&\ll T_0x_1^{2-3\sigma+\epsilon}.
\end{align*}
But for $\sigma\ge 3/4$, $2-3\sigma\le \frac{5}{4}-2\sigma$ and so
\[R\ll T_0x_1^{5/4-2\sigma+\epsilon}.\]
\textbf{Case 2: $N^{4\beta-2}\le T_0$, $M^{4\alpha-2}> T_0$}.
\[M^{1-2\alpha}=x_1^{1-2\sigma}(N^{4\beta-2})^{1/2}\le T_0^{1/2}x_1^{1-2\sigma}.\]
Hence
\[R\ll T_0^{\epsilon}M^{2-2\alpha}\le T_0^{1/2+\epsilon}x_1^{1-2\sigma}M\le T_0x_1^{5/4-2\sigma+\epsilon}\]
since $M\le x_1^{1/2}\le T_0$.

\textbf{Case 3: $N^{4\beta-2}> T_0$, $M^{4\alpha-2}> T_0$}.
\[R\ll (M^{2-2\alpha+\epsilon}N^{2-2\beta+\epsilon})^{1/2}=x_1^{1-\sigma+\epsilon}.\]
But, for $\mu\le 4/(4\sigma-1)+\epsilon$, we have:
\[R\ll x_1^{1-\sigma+\epsilon}\le T_0x_1^{5/4-2\sigma+2\epsilon}.\]
\textbf{Case 4: $N^{4\beta-2}> T_0$, $M^{4\alpha-2}\le T_0$}.

By Lemmas \ref{HuxLV} and \ref{HBR*} we have
\[R\ll \min(T_0^{\epsilon}N^{2-2\beta},T_0^{1+\epsilon}M^{4-6\alpha}),\]
\[R^*\ll M^{1-2\alpha}T_0^{\epsilon}(RM+R^2+R^{5/4}T_0^{1/2})^{1/2}(R^*M+R^4+RR^{*3/4}T_0^{1/2})^{1/2}.\]
But $RM\ge R^2,R^{5/4}T_0^{1/2}$ if $R\le M,M^{4}T_0^{-2}$.

We have
\[R^4\ll N^{6-6\beta}T_0^{1+4\epsilon}M^{4-6\alpha}=T_0^{1+4\epsilon}x_1^{6-6\sigma}M^{-2}.\]
Thus $R^2\ll RMT_0^\epsilon$ if $M\ge x_1^{1-\sigma}T_0^{1/6}$ and $R^{5/4}T_0^{1/2}\ll RMT_0^{3\epsilon}$ if $M\ge x_1^{(1-\sigma)/3}T_0^{1/2-2\epsilon}$. Hence, if
\[M\ge \max(x_1^{(1-\sigma)/3}T_0^{1/2-2\epsilon},x_1^{1-\sigma}T_0^{1/6})\] then
\[R^2+R^{5/4}T_0^{1/2}+RM\ll RMT_0^{3\epsilon}.\]
In this case
\[R^*\ll T_0^{5\epsilon}(M^{4-4\alpha}R+M^{3/2-2\alpha}R^{5/2}+M^{12/5-16\alpha/5}T_0^{2/5}R^{8/5}).\]
We now consider separately each term dominating the RHS.

\textbf{Case 4A: $R^*\ll T_0^{5\epsilon}M^{4-4\alpha}R$}.
\[RR^*\ll T_0^{5\epsilon}M^{4-4\alpha}R^2\ll T_0^{7\epsilon}M^{4-4\alpha}N^{4-4\beta}\ll x_1^{4-4\sigma+7\epsilon}.\]
Since $\mu\le 4/(4\sigma-1)+\epsilon$ we have
\[RR^*\ll x_1^{4-4\sigma+7\epsilon}\ll T_0^2x_1^{9/2-6\sigma+9\epsilon}.\]
It follows that either
\[R\ll T_0x_1^{5/4-2\sigma+2\epsilon}\]
or
\[R^*\ll T_0x_1^{13/4-4\sigma+7\epsilon}.\]
\textbf{Case 4B: $R^*\ll T_0^{5\epsilon}M^{3/2-2\alpha}R^{5/2}$}.
\begin{align*}
R^*&\ll T_0^{5\epsilon}M^{3/2-2\alpha}R^{5/2}\\
&\ll T_0^{6\epsilon}M^{3/2-2\alpha}N^{2-2\beta}\left(T_0^{1+4\epsilon}N^{6-6\beta}M^{4-6\alpha}\right)^{3/8}\\
&\ll x_1^{17(1-\sigma)/4+8\epsilon}M^{-5/4}T_0^{3/8}.
\end{align*}
Hence for $M>T_0^{-1/2}x_1^{(4-\sigma)/5}$ we have
\[R^*\ll T_0x_1^{13/4-4\sigma+8\epsilon}.\]
\textbf{Case 4C: $R^*\ll M^{12/5-16\alpha/5}R^{8/5}T_0^{2/5+5\epsilon}$}.
\begin{align*}
R^*&\ll M^{12/5-16\alpha/5}R^{8/5}T_0^{2/5+5\epsilon}\\
&\ll M^{12/5-16/5\alpha}N^{16/5-16/5\beta}T_0^{2/5+7\epsilon}\\
&\ll x_1^{16(1-\sigma)/5+7\epsilon}T_0^{2/5}M^{-4/5}.
\end{align*}
But for $M>T_0^{-3/4}x_1^{\sigma-1/16}$ we have
\[\ll T_0x_1^{13/4-4\sigma+7\epsilon}.\]
Therefore the Lemma holds, provided that we can always combine the polynomials to ensure that
\[M>x_1^{1-\sigma}T_0^{1/6},x_1^{(1-\sigma)/3}T_0^{1/2-2\epsilon},T_0^{-1/2}x_1^{(4-\sigma)/5},T_0^{-3/4}x_1^{\sigma-1/16}.\]
For $3/4\le \sigma\le 13/16$, $8/5\le \mu\le 4/(4\sigma-1)$ we have
\[x_1^{2/5}\ge x_1^{1-\sigma}T_0^{1/6},x_1^{(1-\sigma)/3}T_0^{1/2},T_0^{-1/2}x_1^{(4-\sigma)/5},T_0^{-3/4}x_1^{\sigma-1/16}\]
and
\[\frac{x_1}{T_0^{1+2\epsilon}}\ge x_1^{1-\sigma}T_0^{1/6},T_0^{-1/2}x_1^{(4-\sigma)/5},T_0^{-3/4}x_1^{\sigma-1/16}\]
If
\[\mu\ge \frac{9}{4+2\sigma}\]
then we have
\[\frac{x_1}{T_0^{1+2\epsilon}}\ge x_1^{(1-\sigma)/3}T_0^{1/2-2\epsilon}.\]
and so
\[M>\min(x_1T_0^{-1},x_1^{2/5})>x_1^{1-\sigma}T_0^{1/6},x_1^{(1-\sigma)/3}T_0^{1/2-2\epsilon},T_0^{-1/2}x_1^{(4-\sigma)/5},T_0^{-3/4}x_1^{\sigma-1/16}\]
as required.

We therefore consider $\mu\le 9/(4+2\sigma)$. Since we are considering $N^{4\beta-2}>T_0$, if $N\le T_0^{3/2+2\epsilon}x_1^{5/4-2\sigma}$ then
\[R\ll T_0^{\epsilon}N^{2-2\beta}\ll \frac{N}{T_0^{1/2-\epsilon}}\ll T_0 x_1^{5/4-2\sigma+3\epsilon}.\]
Therefore we only need to consider $N\ge T_0^{3/2+2\epsilon}x_1^{5/4-2\sigma}$, and so (since $NM=x_1$)
\[\frac{x_1}{T_0^{1+2\epsilon}}\le M\le \frac{x_1^{2\sigma-1/4}}{T_0^{3/2+2\epsilon}}.\]
This means we must have
\[\mu\ge \frac{2}{8\sigma-5}.\]
Since we also have
\[\mu\le \frac{9}{4+2\sigma}\]
we must have $\sigma\ge53/68$. But in this range we can use \eqref{HuxResult} again. This gives 
\[R\ll T_0^{3(1-\sigma)/(3\sigma-1)+\epsilon}\ll T_0x_1^{5/4-2\sigma+\epsilon}\]
provided we have
\[\mu\le\frac{8(3\sigma-2)}{(8\sigma-5)(3\sigma-1)}.\]
This, combined with $\mu\le 9/(4+2\sigma)$ means that we must have $\sigma\le (271-\sqrt{193})/336<53/68$. Therefore we have covered all possible values of $\mu\ge8/5$ and $3/4\le\sigma\le 13/16$. Thus the Lemma holds.
\end{proof}
%\newpage
\subsection{Part 4: $3/4\le\sigma\le 1-10^{-22}$, $\mu$ large}
We now consider the range of $\mu\ge4/(4\sigma-1)+\epsilon$, $3/4\le \sigma\le 1-10^{-22}$. We require separate treatment for $\sigma$ very close to 1.
\begin{lmm}\label{mu}Let $3/4\le \sigma\le 1-10^{-22}$, $\mu\ge 4/(4\sigma-1)+\epsilon$. Then either
\[R\ll x_1^{1-\sigma}(\log{x_1})^{-4k-2}\]
or
\[R^*\ll T_0x_1^{13/4-4\sigma+8\epsilon}.\]
\end{lmm}
\begin{proof}
We assume that $R\ll x_1^{1-\sigma}(\log{x_1})^{-4k-2}$ does not hold. Therefore the results of Lemmas \ref{ZetaBounds} and \ref{Published} apply.

The result follows from published estimates unless $13/16\le\sigma\le 25/28$ and $\mu\le 3/(10\sigma-7)+\epsilon$.

By \eqref{HB3Result} (choosing $\delta=\epsilon$) we have
\[R^*\ll T_0^{(12-12\sigma)/(4\sigma-1)+\epsilon}=T_0T_0^{(13-16\sigma)/(4\sigma-1)+\epsilon}.\]
Since we have $\mu\ge 4/(4\sigma-1)+\epsilon$, if $\sigma\le 13/16$ this gives $R^*\ll T_0x_1^{13/4-4\sigma+\epsilon}$. Therefore without loss of generality we may assume $13/16\le\sigma$.

By \eqref{HB4Result}, if $\sigma\ge 25/28$ then (choosing $\delta=10^{-23}\epsilon$) we have that 
\[R\ll T_0^{(4-4\sigma)/(4\sigma-1)+10^{-23}\epsilon}.\]
Since $\mu\ge 4/(4\sigma-1)+\epsilon$ and $1-\sigma\ge 10^{-22}$, we have for $\sigma\ge 25/28$ that
\[R\ll T_0^{(4/(4\sigma-1)+\epsilon)(1-\sigma)-\epsilon(1-\sigma)+10^{-23}\epsilon}\ll x_1^{1-\sigma}(\log{x_1})^{-4k-2}.\]
Therefore without loss of generality we may assume $\sigma \le 25/28$.

By \eqref{HB2Result} (choosing $\delta=\epsilon/28$) we have
\[R\ll T_0^{(3-3\sigma)/(10-7\sigma)+\epsilon/28}.\]
Therefore if $\mu \ge 3/(10\sigma-7)+\epsilon$ and $13/16\le \sigma\le 25/28$ we have
\[R\ll T_0^{(3/(10\sigma-7)+\epsilon)(1-\sigma)-\epsilon(1-\sigma)+\epsilon/28}\ll x_1^{1-\sigma}(\log{x_1})^{-4k-2}.\]
Therefore without loss of generality we may assume $\mu\le 3/(10\sigma-7)+\epsilon$.

We now consider the remaining case of $13/16\le \sigma\le 25/28$ and $\mu\le 3/(10\sigma-7)+\epsilon$.

We repeatedly combine any pair of polynomials of length $\le T_0^{10^{-24}\epsilon}$ so at most one polynomial has length $\le T_0^{10^{-24}\epsilon}$.

We then pick a polynomial $S_{j_1}$ of length $N_{j_1}\ge T_0^{10^{-24}\epsilon}$ with $\sigma_{j_1}$ maximal.

If $\sigma_{j_1}\le \sigma$ then, since $\sigma$ is an average of the $\sigma_i$, the polynomial $S_{j_2}$ with length $N_{j_2}\le T_0^{10^{-24}\epsilon}$ must exist and have $\sigma_{j_2}\ge \sigma$. In this case we combine the polynomial $S_{j_1}$ and $S_{j_2}$ to produce a polynomial $S$ of length $N$ and size $\sigma'\ge \sigma$.

If $\sigma_{j_1}\ge\sigma$ then we take $S=S_{j_1}$, (and so $N=N_{j_1}$, $\sigma'=\sigma_{j_1}$). 

We consider separately the cases when $N$ is small and $N$ is large.

\textbf{Case 1:}$N\le T_0^{1/(4\sigma-1)}$.

We follow the analysis of Heath-Brown in \cite{HBZeroDensity}[Pages 228-230].

If $N\le T_0^{1/(8\sigma-2)}$ we raise it to a suitable exponent so that the new polynomial has length $M$ with $T_0^{1/(8\sigma-2)}\le M\le T_0^{1/(4\sigma-1)}$. We note that since $N\ge T_0^{10^{-22}\epsilon}$ this exponent is $O(1)$ and so all the coefficients are still $O_\delta(T_0^\delta)$ for every $\delta>0$.

We raise $M$ to different exponents to use in the $R$ and $R^*$ estimates. We let $M_1=M^{k_1}$ which we will use for bound $R$ and we let $M_2=M^{k_2}$ which we will use to bound $R^*$. We choose $k_2$ such that $M^{k_2}\le T_0^{2/(4\sigma-1)}<M^{1+k_2}$, which means $k_2=2$ or $3$ and $T_0^{4/(12\sigma-3)}\le M_2\le T_0^{2/(4\sigma-1)}$. We pick $k_1=k_2$ when $T_0^{1/(3\sigma-1)}\le M_2$ and $k_1=1+k_2$ for $M_2\le T_0^{1/(3\sigma-1)}$. Using Lemma \ref{HuxLV} and recalling that $\sigma'\ge \sigma$ this gives for any $\delta>0$
\[
R\ll_\delta
\begin{cases}
T_0^\delta M_1^{2-2\sigma},\qquad & T_0^{1/(4\sigma-2)}\le M_1\\
T_0^{1+\delta}M_1^{4-6\sigma},&M_1\le T_0^{1/(4\sigma-2)}.\\
\end{cases}\]
If $M_2\le T_0^{1/(3\sigma-1)}$ then $M_1=M_2^{4/3}$ or $M_2^{3/2}$ (depending on whether $k_1=2$ or $3$). In this case, for either value of $k_1$, we get $R\ll_\delta T_0^\delta(TM_2^{16/3-8\sigma}+M_2^{3-3\sigma})\ll T_0^\delta M_2^{3-3\sigma}$. Using this bound is sufficient for our purposes.

Thus, using the above and Lemma \ref{HuxLV} we get the following bound for any $\delta>0$
\[R\ll_\delta
\begin{cases}
T_0^\delta M_2^{2-2\sigma},\qquad & T_0^{1/(4\sigma-2)}\le M_2\le T_0^{2/(4\sigma-1)}\\
T_0^{1+\delta}M_2^{4-6\sigma},&T_0^{1/(3\sigma-1)}\le M_2\le T_0^{1/(4\sigma-2)}\\
T_0^{\delta}M_2^{3-3\sigma},&T_0^{4/(12\sigma-3)}\le M_2\le T_0^{1/(3\sigma-1)}.
\end{cases}\]
We now consider each range of $M_2$ separately.

\textbf{Case 1A: }$T_0^{1/(4\sigma-2)}\le M_2\le T_0^{2/(4\sigma-1)}$.

For $T_0^{1/(4\sigma-2)}\le M_2\le T_0^{2/(4\sigma-1)}$ we have for $\delta=\epsilon/28$ 
\[R\ll M^{2-2\sigma}T_0^{\epsilon/28}\ll T_0^{(4/(4\sigma-1)+\epsilon)(1-\sigma)+\epsilon/28-\epsilon(1-\sigma)}\ll x_1^{1-\sigma}(\log{x_1})^{-4k-2}\]
since $1-\sigma\ge 3/28$ and $\mu\ge 4/(4\sigma-1)+\epsilon$.

\textbf{Case 1B: }$T_0^{1/(3\sigma-1)}\le M_2\le T_0^{1/(4\sigma-2)}$.

For $T_0^{1/(3\sigma-1)}\le M_2\le T_0^{1/(4\sigma-2)}$ we have $R\ll T_0^{1+\epsilon}M_2^{4-6\sigma}$. This means that $R^{5/4}T_0^{1/2}\ll RM_2$ and $R^2\ll RM_2$, so Lemma \ref{HBR*} simplifies to
\[R^*\ll T_0^{2\epsilon}(RM_2^{4-4\sigma}+R^{5/2}M_2^{(3-4\sigma)/2}+R^{8/5}T_0^{2/5}M_2^{(12-16\sigma)/5}).\]
Using $R\ll T_0^{1+\epsilon}M_2^{4-6\sigma}$ this gives
\[R^*\ll T_0^{5\epsilon}(T_0M_2^{8-10\sigma}+T_0^{5/2}M_2^{(23-34\sigma)/2}+T_0^2M_2^{(44-64\sigma)/5}).\]
Since we are considering $\sigma\ge 13/16$, all the exponents of $M_2$ are negative. Thus since $M_2\ge T_0^{1/(3\sigma-1)}$ we have
\[R^*\ll T_0^{5\epsilon}(T_0^{(7-7\sigma)/(3\sigma-1)}+T_0^{(18-19\sigma)/(6\sigma-2)}+T_0^{(34-34\sigma)/(15\sigma-5)}).\]
But for $\sigma\ge 13/16$ we have
\[\max\left(\frac{7-7\sigma}{3\sigma-1},\frac{18-19\sigma}{6\sigma-2},\frac{34-34\sigma}{15\sigma-5}\right)\le 1+\left(\frac{3}{10\sigma-7}\right)\left(\frac{13-16\sigma}{4}\right).\]
This means that we have $R\ll T_0x_1^{13/4-4\sigma+6\epsilon}$ since $\mu\le 3/(10\sigma-7)+\epsilon$.

\textbf{Case 1C: }$T_0^{4/(12\sigma-3)}\le M_2\le T_0^{1/(3\sigma-1)}$.

For $T_0^{4/(12\sigma-3)}\le M_2\le T_0^{1/(3\sigma-1)}$ we have $R\ll M_2^{3-3\sigma}T_0^{\epsilon}$. Therefore $R\ll M_2$ and Lemma \ref{HBR*} simplifies to
\begin{align*}
R^*&\ll M_2^{1-2\sigma}T_0^{\epsilon}(R^{1/2}R^{*1/2}M_2+R^{5/2}M_2^{1/2}+RR^{*3/8}M_2^{1/2}T_0^{1/4}\\
&\qquad\qquad\qquad+R^{5/8}R^{*1/2}M_2^{1/2}T_0^{1/4}+R^{21/8}T_0^{1/4}+R^{9/8}R^{*3/8}T_0^{1/2}).
\end{align*}
We note that using the trivial bound $R^*\ll (\log{T})R^3$ we have $RR^{*3/8}M_2^{1/2}T_0^{1/4+\epsilon}\gg R^{5/8}R^{*1/2}M_2^{1/2}T_0^{1/4}$. We can therefore drop the fourth term at the cost of a factor $\ll T_0^\epsilon$. This yields
\begin{align*}
R^*&\ll T_0^{4\epsilon}(RM_2^{4-4\sigma}+R^{5/2}M_2^{(3-4\sigma)/2}+R^{8/5}M_2^{(12-16\sigma)/5}T_0^{2/5}\\
&\qquad\qquad\qquad+R^{21/8}T_0^{1/4}M_2^{1-2\sigma}+R^{9/5}M_2^{(8-16\sigma)/5}T_0^{4/5}).
\end{align*}
Substituting in $R\ll M_2^{3-3\sigma+\epsilon}$ we get
\begin{align*}
R^*&\ll T_0^{7\epsilon}(M_2^{7-7\sigma}+M_2^{(18-19\sigma)/2}+M_2^{(36-40\sigma)/5}T_0^{2/5}\\
&\qquad\qquad\qquad+M_2^{(71-79\sigma)/8}T_0^{1/4}+T_0^{4/5}M_2^{(35-43\sigma)/5}).
\end{align*}
For $13/16\le \sigma\le 25/28$ the first four terms always have positive exponents of $M_2$. Thus, using $T_0^{4/(12\sigma-3)}\le M_2\le T_0^{1/(3\sigma-1)}$ we get
\begin{align*}
R^*&\ll T_0^{7\epsilon}(T_0^{(7-7\sigma)/(3\sigma-1)}+T_0^{(18-19\sigma)/(6\sigma-2)}+T_0^{(34-34\sigma)/(15\sigma-5)}\\&\qquad\qquad+T_0^{(69-73\sigma)/(24\sigma-8)}+T_0^{(31-31\sigma)/(15\sigma-5)}+T_0^{(128-124\sigma)/(60\sigma-15)}).
\end{align*}
But for $13/16\le \sigma\le 25/28$ we have
\begin{align*}
\max&\Biggl(\frac{7-7\sigma}{3\sigma-1},\frac{18-19\sigma}{6\sigma-2},\frac{34-34\sigma}{15\sigma-5},\\
&\qquad\frac{69-73\sigma}{24\sigma-8},\frac{31-31\sigma}{15\sigma-5},\frac{128-124\sigma}{60\sigma-15}\Biggr)\le 1+\left(\frac{13-16\sigma}{4}\right)\left(\frac{3}{10\sigma-7}\right).
\end{align*}
Thus $R^*\ll T_0x_1^{13/4-4\sigma+8\epsilon}$ since $\mu\le \frac{3}{10\sigma-7}+\epsilon$.

Putting these estimates together covers all possible values of $\mu\ge 4/(4\sigma-1)+\epsilon$ and $3/4\le \sigma\le 1-10^{-22}$.

\textbf{Case 2: }$N\ge T_0^{1/(4\sigma-1)}$.

We will choose $k\ge 7$ so that \eqref{MuPieceBound} implies that for $i\le k$ we have
\[N_i\le (3x)^{1/7}\le T_0^{1/3}\le T_0^{1/(4\sigma-1)}.\]
Therefore the polynomial selected with this `long' length must either be one with all coefficients $1$ or $\log(n)$ and length $N=N_{j_1}$, or it must be a the combination of such a polynomial with another polynomial of length $N_{j_2}\le T_0^{10^{-24}\epsilon}$ (so $N=N_{j_1}N_{j_2}$).

By Lemma \ref{ZetaBounds} either
\[R\ll x_1^{1-\sigma}(\log{x_1})^{-4k-2}\]
or
\[R\ll_\delta T_0^{2+\delta}N_{j_1}^{6-12\sigma_{j_1}}\]
for any $\delta>0$.

Without loss of generality we assume $R\ll_\delta T_0^{2+\delta}N_{j_1}^{6-12\sigma_{j_1}}$.

If $N=N_{j_1}N_{j_2}$ then since $N_{j_2}\ll T_0^{10^{-24}\epsilon}$ we have
\begin{align*}
R&\ll R(T_0^{6\times10^{-24}\epsilon} N_k^{6-12\sigma_k})\\
&\ll T_0^{2+10^{-23}\epsilon}N^{6-12\sigma'}.
\end{align*}
The same result clearly holds if $N=N_{j_1}$.

Thus, since $N\ge T_0^{1/(4\sigma-1)}$, we have
\begin{align*}
R&\ll T_0^{2-6(2\sigma-1)/(4\sigma-1)+10^{-23}\epsilon}\\
&\ll T_0^{(1-\sigma)(4/(4\sigma-1)+\epsilon)-\epsilon(1-\sigma)+10^{-23}\epsilon}\\
&\ll x_1^{(1-\sigma)}(\log{x_1})^{-4k-2}
\end{align*}
since
\[\mu>\frac{4}{4\sigma-1}+\epsilon,\qquad 1-\sigma\ge 10^{-22}.\]
\end{proof}

%\newpage
\subsection{Part 5: $1-10^{-22}\le \sigma\le 1$}
We now consider the final range, when $1-10^{-22}\le \sigma \le 1$. We split the argument into two cases - when $\sigma$ is exceptionally close to one, and so we can use Vinogradov's bound, and the remaining case.
\begin{lmm}\label{SClose1}
Let $1-10^{-22}\le\sigma\le 1$, $\mu\ge 4/(4\sigma-1)+\epsilon$. Then we have
\[R\ll x_1^{1-\sigma}(\log{x_1})^{-4k-2}\]
\end{lmm}
\begin{proof}
We recall that
\[\eta=\eta(x_1)=C_0(\log{x_1})^{-2/3}(\log\log{x_1})^{-1/3}\]
for some constant $C_0>0$ from \eqref{EtaDef}.

\textbf{Case 1: $\sigma\le 1-\eta$}

We consider separately the case when all polynomials are small.

\textbf{Case 1A:} $N_i\le x_1^{1/k}$ $\forall$ $i$

We repeatedly combine pairs of polynomials of length $\le x_1^{1/2k}$. Thus without loss of generality we assume all polynomials have length $\in$ $[x_1^{1/2k},x_1^{1/k}]$, except for possibly one polynomial with length $\le x_1^{1/2k}$. We combine this polynomial with one of the remaining ones, so all polynomials have length $\in$ $[x_1^{1/2k},x_1^{3/2k}]$.

Since we have combined at most $2k$ polynomials, all remaining polynomials must have coefficients which are $\ll (\log{x_1})d_{2k}(n)$.

We pick a polynomial which has $\sigma_i\ge \sigma$. We raise this polynomial to an exponent so that it has length $Y$ with $T_0^{9/16}\le Y\le T_0^{5/8}$. This is possible if the polynomial's original length was $\le T_0^{1/16}$. This is the case provided $x_1^{3/2k}\le T_0^{1/16}$. Thus if we now choose
\begin{equation}\label{KChoice}
k=60
\end{equation}
then this is satisfied, since $\mu\le 19/9$. We note that we raise the polynomial to an exponent $\le 75$, and so all the coefficients of this new polynomial of length $Y$ are $\ll (\log{x_1})^{75}d_{9000}(n)$.

Using Lemma \ref{HuxLV} we have
\begin{align*}
R&\ll (\log{x_1})^2(Y^{2-2\sigma_i}+T_0Y^{4-6\sigma_i})\left(1+\frac{\sum_Y^{2Y}(\log{x_1})^{150}d_{9000}(n)^2}{Y}\right)^3\\
&\ll (\log{x_1})^{60000}(Y^{2-2\sigma}+T_0Y^{4-6\sigma})
\end{align*}
But for $\sigma\ge 17/18$ we have $T_0^{9(4\sigma-2)/16}\ge T_0$ and so $Y^{4\sigma-2}\ge T_0$. This means that $Y^{2-2\sigma}\ge T_0Y^{4-6\sigma}$. Hence
\begin{align*}
R&\ll Y^{2-2\sigma}(\log{x_1})^{60000}\\
&\ll T_0^{5(1-\sigma)/4}(\log{x_1})^{60000}\\
&\ll x_1^{15(1-\sigma)/16}(\log{x_1})^{60000}\\
&\ll x_1^{1-\sigma}\exp(-C_0(\log{x_1})^{1/3}(\log\log{x_1})^{-2/3}/16+60000\log\log{x_1})\\
&\ll x_1^{1-\sigma}(\log{x_1})^{-4k-2}
\end{align*}
Since $\sigma\le 1-\eta$, $\mu\ge 4/3$.

\textbf{Case 1B:}$\exists$ $N_j>x_1^{1/k}$

By \eqref{MuPieceBound} this polynomial must have all coefficients 1 or $\log(n)$. We first consider the case when all coefficients are $1$.

By Van-der-Corput's method of exponential sums (see the proof of \cite{Titchmarsh}[Theorem 5.14], for example) we have for any $l\in\mathbb{Z}$ and $\alpha=1-l/(2^l-2)$
\[\sum_N^{2N}n^{-\alpha-it}\ll T_0^{1/(2^l-2)+\epsilon}\]
Thus
\begin{align*}
|S_j|=\left|\sum_{N_j}^{2N_j}n^{-c-it}\right|&\ll N_j^{-c+\alpha}T_0^\epsilon\sup_{N\in [N_j/2,N_j]}\left|\sum_{N}^{2N}n^{-\alpha-it}\right|\\
&\ll N_j^{-l/(2^l-2)}T_0^{1/(2^l-2)+2\epsilon}.
\end{align*}
The same result folds for a polynomial with all coefficients $(\log{n})$ by partial summation.

We have $N_j>x_1^{1/k}>T_0^{1/k}$ so
\[|S_j|\ll N_j^{-(l-k)/(2^l-2)+2k\epsilon}\]
Thus, choosing $l=k+1$ we have
\[|S_j|\ll N_j^{-2^{-k-1}+2k\epsilon}\]
and so (for $\epsilon\le 2^{-k-3}/k$) we must have 
\[\sigma_j\le 1-2^{-k-1}+2k\epsilon\le 1-2^{-k-2}\]
since $|S_i|\ll N_i^{-1+\sigma_i}$. But
\begin{align*}
x_1^\sigma&=\prod_iN_i^{\sigma_i}\\
&\ll N_j^{\sigma_j}\prod_{i\ne j}N_i\\
&=x_1N_j^{\sigma_j-1}\\
&\ll x_1(x_1^{1/k})^{-2^{-k-2}}\\
& \ll x_1^{1-2^{-k-2}/k}.
\end{align*}
Hence
\[\sigma\le 1-\frac{1}{k2^{k+2}}\]
or $R=0$.

By \eqref{KChoice} we have $k=60$, so this means that
\[\sigma\le 1-10^{-22}\qquad\text{or}\qquad R=0\]
and so we are done.

\textbf{Case 2: $\sigma\ge1-\eta$}

By the same argument as in Case 1 of Lemma \ref{LargePolys}, we get the result \eqref{VingradovSigmaNearOne}
\[R=0\qquad\text{or}\qquad\sigma_{j'}\le 1-3\eta/2\]
for any polynomial with $N_{j'}\gg x_1^{1/40k}$.

Thus either $R=0\ll T_0x_1^{2\sigma-5/4+\epsilon}$ or (since there are only $2k$ polynomials $S_i$) 
\[x_1^{\sigma}=\prod_{i=1}^{2k}N_i^{\sigma_i}\le \left(\prod_{\substack{1\le i\le 2k\\N_i\le x_1^{1/40k}}}N_i\right)\left(\prod_{\substack{1\le i\le 2k\\N_i\ge x_1^{1/40k}}}N_i^{1-3\eta/2}\right)\le x_1^{1-\eta}\]
and so we must have $\sigma<1-\eta$.
\end{proof}
This covers all the different cases, and so the main result holds.

\section{Acknowledgment}
I would like to thank my supervisor, Prof. Heath-Brown, for suggesting this problem, for providing a large number of helpful comments as well as encouragement. I would also like to thank Christian Elscholtz and Yoichi Motohashi of making me aware of the existing result of Peck in this subject.
\newpage
\section{Comments and Further Work}
Using the above argument we obtain the best possible result in some sense. Without improving the existing estimates for large values of Dirichlet Polynomials, it appears an exponent of $5/4+\epsilon$ is the smallest obtainable using the method presented.

The critical case in the argument appears to be when $\sigma_i=3/4$ $\forall i$. If there are 4 polynomials all of equal length (i.e. length $x_1^{1/4}$) then throughout the range $x_1^{11/20}\le\tau\le x_1^{5/8}$ Proposition \ref{mainprop} fails to hold for any exponent $\le 5/4$ using the estimates for the frequency of large values of Dirichlet Polynomials when $\sigma=3/4$. In this region we use the strongest known such bounds, and so an improvement to the result would require a stronger large values estimate when $\sigma=3/4$. Improving the estimates at $\sigma=3/4$ appears to be difficult. Several improvements have been made to Montgomery's and Huxleys estimates given in Lemmas \ref{MontMV} and \ref{HuxLV} for other ranges of $\sigma$, but $\sigma=3/4$ appears to be the hardest to improve.  The bounds given are also tight in the region $\sigma\ge 25/28$, but it appears for $\sigma$ large there is more flexibility to improve the large value estimates. For $\sigma$ large there are various stronger estimates for $R$ which have not been employed here.

The bound obtained is tight in $\tau$ only for $x^{11/20}\le\tau\le x^{5/8}$ or $\tau=x^{1/2}$, which is far from the full range. Therefore the above argument implies a slightly stronger result, where we have
\[\sum_{p_n\le x}f(d_n)\ll x^{5/4+\epsilon}\]
for some function $f(t)\ge t^2$ and $f(t)\ge t^{2+\epsilon_1}$ for the range when $\tau\le x^{1/2-\epsilon'}$, or $x^{1/2+\epsilon'}\le\tau\le x^{11/20-\epsilon'}$, or $\tau\ge x^{5/8+\epsilon'}$ for some $\epsilon'>0$.

It might be possible to improve the result by combining the method with sieve ideas. This was successfully employed by Baker, Harman and Pintz \cite{BakerEtAl} in their result $d_n\ll p_n^{21/40}$. Employing a suitable sieve might enable one to avoid the critical case in our argument when $\sigma=3/4$ and $x^{11/20}\le\tau\le x^{5/8}$, thereby enabling us to improve on the overall result.

Yu \cite{Yu} employed a large double sieve to the problem when assuming the Lindel\"of hypothesis. Although it appears that following exactly the method he employed does not improve the exponent when the Lindel\"of assumption is dropped, the large double sieve could potentially aid the argument in another form. Following Yu's argument but using the bound $\zeta(1/2+it)\ll t^{\theta}$ gives a bound which approaches $x^{2+\epsilon}$ continuously as $\theta$ approaches 0. Using the best existing estimates of the order of $\zeta(1/2+it)$ (which are slightly smaller than 1/6) fails to produce an exponent better than $5/4$, and the argument does not seem to avoid the complications of the critical case in our argument.

%\bibliographystyle{acm}
%\bibliography{bibliography}
\end{document}